\documentclass[12pt,british,a4paper]{amsart}

\usepackage{babel,amssymb,url,stmaryrd,mathtools,fixltx2e}
\usepackage{paralist,stackrel}

\numberwithin{equation}{section}
\mathtoolsset{mathic=true}
\setdefaultenum{\upshape (i)}{}{}{}

\theoremstyle{plain}
\newtheorem{theorem}[equation]{Theorem}
\newtheorem{proposition}[equation]{Proposition}
\newtheorem{lemma}[equation]{Lemma}
\newtheorem{corollary}[equation]{Corollary}

\theoremstyle{definition}
\newtheorem{definition}[equation]{Definition}

\theoremstyle{remark}
\newtheorem{Example}[equation]{Example}
\newtheorem{remark}[equation]{Remark}

\DeclareRobustCommand{\qee}{%
  \ifmmode \mathqee
  \else
    \leavevmode\unskip\penalty9999 \hbox{}\nobreak\hfill
    \quad\hbox{\qeesymbol}%
  \fi
}
\newcommand{\mathqee}{\quad\hbox{\qeesymbol}}
\newcommand{\qeesymbol}{\ensuremath\diamondsuit}

\newcommand{\Lie}[1]{\operatorname{\textsl{#1}}}
\newcommand{\lie}[1]{\operatorname{\mathfrak{#1}}}
\newcommand{\GL}{\Lie{GL}}

\newcommand{\sln}{\lie{sl}}

\newcommand{\su}{\lie{su}}

\newcommand{\bmf}{\lie b}
\newcommand{\cf}{\lie c}
\newcommand{\kf}{\lie k}
\newcommand{\n}{\lie n}
\newcommand{\p}{\lie p}
\newcommand{\q}{\lie q}
\newcommand{\tf}{\lie t}
\newcommand{\z}{\lie z}
\newcommand{\SL}{\Lie{SL}}
\newcommand{\SU}{\Lie{SU}}
\newcommand{\Un}{\Lie{U}}

\newcommand\C{{\mathbb C}}
\newcommand{\HH}{{\mathbb H}}

\newcommand{\R}{{\mathbb R}}

\newcommand{\impl}{{\textup{impl}}}
\newcommand{\hks}{{\textup{hks}}}

\newcommand{\symp}{{\sslash}} 
\newcommand{\hkq}{{\sslash\mkern-6mu/}}

\DeclareMathOperator{\Hom}{Hom}

\DeclareMathOperator{\im}{im}

\DeclareMathOperator{\Spec}{Spec}
\DeclareMathOperator{\Stab}{Stab}
\DeclareMathOperator{\tr}{tr}

\DeclarePairedDelimiter{\abs}{\lvert}{\rvert}

\newcommand{\eqbreak}[1][2]{\\&\hskip#1em}

\begin{document}

\title{Implosions and hypertoric geometry}

\author{Andrew Dancer}
\address[Dancer]{Jesus College\\
Oxford\\
OX1 3DW\\
United Kingdom} \email{dancer@maths.ox.ac.uk}

\author{Frances Kirwan}
\address[Kirwan]{Balliol College\\
Oxford\\
OX1 3BJ\\
United Kingdom} \email{kirwan@maths.ox.ac.uk}

\author{Andrew Swann}
\address[Swann]{Department of Mathematics\\
Aarhus University\\
Ny Munkegade 118, Bldg 1530\\
DK-8000 Aarhus C\\
Denmark\\
\textit{and}\\
CP\textsuperscript3-Origins,
Centre of Excellence for Cosmology and Particle Physics Phenomenology\\
University of Southern Denmark\\
Campusvej 55\\
DK-5230 Odense M\\
Denmark} \email{swann@imf.au.dk}

\subjclass[2000]{53C26, 53D20, 14L24}

\maketitle

\setcounter{section}{-1}

\begin{center}
\itshape Dedicated to Professor C. S. Seshadri on the occasion \\ of his 80th birthday
\end{center}
\section{Introduction}
\label{sec:introduction}
Hyperk\"ahler manifolds occupy a special position at the intersection
of Riemannian, symplectic and algebraic geometry. A hyperk\"ahler
structure involves a Riemannian metric, as well as a triple of
complex structures satisfying the quaternionic relations. Moreover we require
that the metric is K\"ahler with respect to each complex structure, so we 
have a triple (in fact a whole two-sphere) of symplectic forms. Of course,
there is no Darboux theorem in hyperk\"ahler geometry because the
metric contains local information. However, many of the
constructions and results of symplectic geometry, especially
those related to moment maps, do have analogues
in the hyperk\"ahler world. The prototype is the hyperk\"ahler quotient
construction \cite{HKLR}, and more recent examples include
hypertoric varieties \cite{BD} and cutting \cite{DSmod}.

In this article we shall explore a hyperk\"ahler analogue of
Guillemin, Jeffrey and Sjamaar's construction of \emph{symplectic
  implosion} \cite{GJS}.  This may be viewed as an
abelianisation procedure: given a symplectic manifold \( M \) with a
Hamiltonian action of a compact group \( K \), the implosion \(
M_\impl \) is a new symplectic space with an action of the maximal
torus \( T \) of \( K \), such that the symplectic reductions of \(
M_\impl \) by \( T \)  agree with the reductions of \( M \) by \( K
\). However the implosion is usually not smooth but is a singular space
with a stratified symplectic structure.
 The implosion of the cotangent bundle \( T^*K \) acts as a universal
object here; implosions of general Hamiltonian \( K \)-manifolds may be defined
using the symplectic implosion \(
(T^*K)_\impl \). This space also has an algebro-geometric description as 
the geometric invariant theory quotient of
 \( K_\C \) by a maximal unipotent subgroup~\( N \).

In \cite{DKS} we introduced a hyperk\"ahler analogue of the universal
implosion in the case of \( \SU(n) \) actions.
The construction proceeds via quiver diagrams, and
produces a stratified hyperk\"ahler space \( Q \).
 The hyperk\"ahler strata can be described in terms of open sets
in complex symplectic quotients of the cotangent bundle 
of \( K_\C=\SL(n, \C) \) by subgroups containing
commutators of parabolic subgroups.  There is a maximal torus action,
and hyperk\"ahler quotients by this action 
 yield not single complex coadjoint orbits but
rather their canonical affine completions which are Kostant varieties.

In this article, we shall develop some of the ideas of \cite{DKS},
focusing on some aspects, such as toric geometry and gauge theory
constructions, which may generalise to the case of an arbitrary
compact group \( K \). In particular, we shall show the existence
in the case \( K=SU(n) \) of a hypertoric variety inside the implosion,
which has a natural description in terms of quivers. This 
is a hyperk\"ahler analogue of the result of \cite{GJS} that the
universal symplectic implosion \(
(T^*K)_\impl \) naturally contains the toric variety
associated to a positive Weyl chamber for \( K \).
 
The layout of the paper is as follows. In \S1 we review the theory of symplectic
implosion described in \cite{GJS}, and in \S2 we recall how hyperk\"ahler 
implosion for \( K=\SU(n) \) is introduced in \cite{DKS}.  In \S3 we recall some of the theory 
of hypertoric varieties and describe a hypertoric variety which maps naturally to the universal hyperk\"ahler implosion \( Q \) for \( K=\SU(n) \). In \S4 we recall the stratification given in \cite{DKS} of \( Q \) into strata which are hyperk\"ahler manifolds, and in \S5 we refine this stratification to obtain strata \( Q_{[\sim,\mathcal{O}]} \) 
which are not hyperk\"ahler but which reflect the group structure of \( K=\SU(n) \) and
can be indexed in terms of Levi subgroups and nilpotent orbits in the complexification \( K_\C \) of \( K \). In \S6, \S7 and \S8 we use Jordan canonical form to describe open subsets of the refined strata by putting their quivers into standard forms.  
Finally in \S9 we explore briefly the relationship between the finite-dimensional picture of the universal hyperk\"ahler implosion \( Q \) for \( K=\SU(n) \) and an infinite-dimensional point of view involving the Nahm equations.

\subsubsection*{Acknowledgements.} The work of the second author was
supported by a Senior Research Fellowship of the Engineering and
Physical Sciences Research Council (grant number GR/T016170/1) during
much of this project. The third author is partially supported by the
Danish Council for Independent Research, Natural Sciences.

\section{Symplectic implosion}
\label{sec:symplectic-implosion}

Our study of hyperk\"ahler implosion in \cite{DKS} was motivated by
the theory of symplectic implosion, due to Guillemin,
Jeffrey and Sjamaar \cite{GJS}.  For this we start with a symplectic
manifold \( M \) with a Hamiltonian symplectic action of a compact Lie
group \( K \) with maximal torus \( T \).
If \( \lambda \) is a central element of \( \kf^* \) the
symplectic reduction \( M \symp_\lambda^s K \) is 
\( \mu^{-1}(\lambda)/K \) where \( \mu : M \rightarrow \kf^* \) is the moment
map for the action of \( K \) on \( M \). For a general element \( \lambda \in \kf^* \),
we define the symplectic reduction \( M \symp_\lambda^s K \)
to be the space \( (M
\times \mathsf O_{-\lambda}) \symp_0^s K \), where \( \mathsf
O_\lambda \) is the coadjoint orbit of \( K \) through \( \lambda \) with the
standard Kirillov-Kostant-Souriau symplectic structure.
 This reduction may be
identified with \( \mu^{-1}(\lambda)/\Stab_K(\lambda) \) where
 \( \Stab_K(\lambda) \) is the stabiliser of \( \lambda \) under the 
coadjoint action of \( K \).

 The imploded space \( M_\impl \) is a stratified symplectic space with a 
Hamiltonian action of the maximal torus \( T \) of \( K \), such that
\begin{equation}
  M \symp_\lambda^s K   = M_\impl \symp^s_\lambda T
\end{equation}
for all \( \lambda \) in the closure \( \tf_{+}^* \) of a fixed positive Weyl
chamber in \( \tf^* \).

The key example is the implosion of the
  cotangent bundle \( T^*K \). Now \( T^*K \) carries a \( K \times K \) action,
which we can think of as commuting left and right actions of \( K \). The
left action is \( (k, \xi) \mapsto (hk,\xi) \) while the right action is
\( (k, \xi) \mapsto (kh^{-1}, Ad(h).\xi) \). 
The moment maps for the left and right actions are 
\begin{equation*}
(k, \xi) \mapsto -Ad(k).\xi
\end{equation*}
and
\begin{equation*}
(k, \xi) \mapsto \xi
\end{equation*}
respectively. 
We shall implode \( T^*K \) with respect to the right action.
Explicitly, \( (T^*K)_\impl \) is obtained from \( K \times \tf_{+}^*
\), by identifying \( (k_1, \xi) \) with \( (k_2, \xi) \) if \( k_1,
k_2 \) are related by the action of an element of the commutator
subgroup of \( \Stab_{K}(\xi)\). Thus if \( \xi \) is in the interior of
the chamber, its stabiliser is a torus and no collapsing occurs, and  
 an open dense subset of \(
(T^*K)_\impl \) is just the product of \( K \) with the interior of
the Weyl chamber.
Now symplectic reduction 
by the right action of \( T \) at level \( \lambda \)
(in the closed positive Weyl chamber)
will fix \( \xi \) to be \( \lambda \), and collapse by the product of \( T \)
with the commutator subgroup of \( \Stab_{K}(\lambda) \), which is equivalent to
collapsing by \( \Stab_{K}(\lambda) \). Now we have
\begin{equation*}
  (T^*K)_\impl\symp_\lambda^s T = K/\Stab_{K}(\lambda)=\mathsf O_\lambda =
(T^*K) \symp_{\lambda}^s K
\end{equation*}
as required.
 \(
(T^*K)_\impl \) inherits  a Hamiltonian \( K \times T \)-action from
the  Hamiltonian \( K \times K \)-action on \( T^*K \).  This gives us a
universal implosion, in the sense that the implosion \( M_\impl \) of
a general symplectic manifold \( M \) with a Hamiltonian \( K
\)-action can be obtained as the symplectic reduction \( (M \times
(T^*K)_\impl) \symp_0^s K \).

It is also shown in \cite{GJS} that the implosion \( (T^*K)_{\impl} \)
may be embedded in the complex affine space
 \( E = \oplus V_{\varpi} \), where \( V_{\varpi} \) is the 
\( K \)-module with highest weight \( \varpi \).  
and we take the sum over a minimal generating set
for the monoid of dominant weights. 
  We denote a highest weight vector of \( V_{\varpi} \) by \(
v_{\varpi} \). In this picture,
 the symplectic implosion may be realised as the
closure \( \overline{K_\C v} \), where \( v = \sum v_\varpi \) is the
sum of the highest weight vectors, and \( K_{\C} \) denotes the complexification
of \( K \).

In terms of the Iwasawa decomposition \(
K_\C=KAN \) we have that the maximal unipotent subgroup
\( N \) is the stabiliser of \( v \),
so an open dense set in the implosion is \( K_\C v = K_{\C}/N \). Taking the
closure gives lower-dimensional strata in the implosion, which
may be identified with quotients \( K_{\C}/[P,P] \) where \( P \) ranges over
parabolic subgroups of \( K_{\C} \). Of course, taking \( P \) to be the Borel \( B \)
gives the top stratum \( K_{\C}/N = K_{\C} /[B,B] \).
In fact the full implosion may be identified with the Geometric Invariant 
Theory (GIT) quotient of \( K_{\C} \) by the nonreductive group \( N \):
\begin{equation*}
  K_\C \symp N = \Spec(\mathcal{O}(K_\C)^N),
\end{equation*}
This may also be viewed as the canonical affine completion of the quasi-affine
variety \( K_\C / N \). (We refer to \cite{DK} for background on 
nonreductive GIT quotients).

Using the Iwasawa decomposition as above, and recalling that \( T_{\C}=TA \),
we see that \(
\overline{K_\C v} = \overline{KAv} = K (\overline{T_\C v}) \), 
the sweep under the
compact group \( K \) of a toric variety \( X= \overline{T_\C v} \).
As \( T_{\C} \) normalises \( N \), we have that \( N \) stabilises every point in \( X \);
in fact \( X \) is the fixed point set \( E^N \) for the action of \( N \) on the
vector space \( E \). The action of the compact torus \( T \) defines
a moment map \( \mu_{T} : X \rightarrow \tf^* \) whose image is (minus) \( \tf_{+}^* \),
so \( -\tf_{+}^* \) is the Delzant polytope for the toric variety \( X \).
Equation (6.6) in \cite{GJS} defines a \( T \)-equivariant 
map \( s : \tf_{+}^* \rightarrow X \) which is a section for \( -\mu_{T} \).
The map \( s \) extends to a \( K \times T \)-equivariant map \( K \times \tf_{+}^*
\rightarrow KX \), which induces a homeomorphism from \( (T^*K)_{\impl} \)
onto \( \overline{K_\C v} \)

Recall that the moment map for the left \( K \) action on \( T^*K \) is
\begin{equation*}
\mu_{K} : (k, \xi) \mapsto -Ad(k).\xi
\end{equation*}
Note that two points \( (k_1, \xi), (k_2, \xi) \) in \( T^*K \) 
with the same \( \kf^* \) coordinate have the same
image under \( \mu_{K} \) if and only if \( k_1 k_2^{-1} \in {\rm Stab}_{K}(\xi) \).
In particular two points of \( K \times \tf^* \) which are identified in the 
implosion will have the same image under \( \mu_K \), so this map descends
to the implosion.

We have a commutative diagram
\smallskip
\begin{equation*}
\begin{array}{ccccc}
(T^*K)_{\impl}  &  \stackrel{\mu_{K}}{\longrightarrow} & \kf^* & \rightarrow &  
\kf^*/K \\
{
\downarrow} &       &                &            & ||  \\
X              &   \stackrel{-\mu_{T}}{\longrightarrow}   & \tf^* &   \rightarrow  &  \tf^* / W
\end{array}
\end{equation*}
where the left vertical arrow is induced by \( (k,\xi) \mapsto s(\xi) \), and the
rightmost arrow in each row is the obvious quotient map.

\medskip
In \cite{DKS} we introduced a new model for the symplectic implosion for \( K=SU(n) \), in 
terms of \emph{symplectic quivers}. These are
diagrams
\begin{equation}
  \label{eq:symplectic}
  0 = V_0 \stackrel{\alpha_0}{\rightarrow}
  V_1 \stackrel{\alpha_1}{\rightarrow}
  V_2 \stackrel{\alpha_2}{\rightarrow} \dots
  \stackrel{\alpha_{r-2}}{\rightarrow} V_{r-1}
  \stackrel{\alpha_{r-1}}{\rightarrow} V_r = \C^n.
\end{equation}
where \( V_i \) is a vector space of dimension \( n_i \). The group
\( \prod_{i=1}^{r-1} \SL(V_i) \) acts on quivers by
\begin{align*}
  \alpha_i &\mapsto g_{i+1} \alpha_i g_i^{-1} \quad (i = 1,\dots, r-2),\\
  \alpha_{r-1} &\mapsto \alpha_{r-1} g_{r-1}^{-1}.
\end{align*}
There is also of course a commuting action of \( \GL(n,\C) = \GL(V_r)
\) by left multiplication of \( \alpha_{r-1} \).
We considered the GIT quotient of the space of quivers by
\( \prod_{i=1}^{r-1} \SL(V_i) \), focusing particularly on the full flag case when 
\( n_i =i \) for all \( i \). It turns out that such a quiver lies in a closed orbit
if and only if, for each \( i \) we have
\begin{compactenum}
  \item \( \alpha_i \) is injective, or
  \item \( V_i = \im\alpha_{i-1} \oplus \ker\alpha_i \).
  \end{compactenum}
We may now decompose \( \C^i = \ker \alpha_i \oplus \C^{m_i} \),
where \( \C^{m_i} = \C^i \) if \( \alpha_i \) is injective and we take
\( \C^{m_i} = \im \alpha_{i-1} \) otherwise. This defines a
decomposition of the quiver into two subquivers; for one subquiver
the maps are all injective while for the other they are all zero.
We may therefore focus on the injective quiver. 
As explained in \S 4 of \cite{DKS}, we may contract any
edges of this quiver where the maps are isomorphisms. More precisely,
if \( m_i = m_{i-1} \) then we have \( m_i \leq i-1 < i \), so we actually
have a \( GL(m_i) \) action on \( \C^{m_i} \) and the isomorphism
\( \C^{m_{i-1}} \rightarrow \C^{m_i} \) may be set to be the identity, so
this edge of the quiver may be removed.
After this process the dimensions of the spaces in the
injective quiver are given by a strictly increasing
sequence of integers ending with \( n \).

The upshot is that we have a stratification of the GIT
quotient by \( \prod_{i=2}^{n-1} \SL(i) \) of the space of full flag
quivers. There are \( 2^{n-1} \) strata, indexed by the strictly
increasing sequences of positive integers ending with \( n \), or
equivalently by the ordered partitions of \( n \). Moreover, 
the injectivity property makes it easy to analyse the structure
of each stratum. For we may now use the  action of \( \prod_{i=2}^{n-1} \SL(i) \)
and \( SL(n) \) to put the \( \alpha_i \) into a standard form where
all entries are zero except for the \( (j,j) \) entries (\( j=1,\ldots, m_i) \),
which equal \( 1 \). The freedom involved in putting the \( \alpha_i \) into
this standard form is exactly an element of the commutator \( [P,P] \),
where \( P \) is the parabolic subgroup of \( \SL(n) \)
corresponding to the ordered partition of \( n \).
We conclude that the strata can be identified with
\( \SL(n)/[P,P] \).
In fact, the full GIT quotient
may be identified with the symplectic implosion for \( SU(n) \) and the strata
are just the strata of the implosion discussed above.

We may also realise the toric structure discussed above in this model.
For we can instead put \( \alpha_i \) into a slightly different standard form where
the \( (j,j) \) entries (\( j=1,\ldots, m_i) \) can 
now equal a nonzero scalar \( \sigma_i \), not necessarily \( 1 \). 
This standard form is now preserved by an element of the parabolic \( P \), and
the \( \sigma_i \) define an algebraic torus of dimension given by the length
of the injective quiver (equivalently, we are considering the 
fibration \( T^{r-1}_{\C} \rightarrow SL(n)/[P,P] \rightarrow SL(n)/P \) for
each stratum). These tori fit together to form the toric variety.
Now the generalised flag variety \( SL(n)/P \) is a homogeneous
space for the compact group \( SU(n) \), so we see again that the sweep
of the toric variety under the \( SU(n) \) action is the full implosion.
Alternatively, we may allow the entries \( \sigma_i \) to depend on
\( j \) as well as \( i \). Now we have an action on such
configurations of the product of the
maximal tori of \( SL(m_i) \), and this action may be used to bring the
quiver into the form above where \( \sigma_i \) depends only on \( i \). 
\section{Hyperk\"ahler implosion}
\label{sec:towards-hyperk-impl}

In \cite{DKS} a hyperk\"ahler analogue of the symplectic implosion was introduced
for the group \( K= SU(n) \). Motivated by the quiver model for symplectic implosion
described in the preceding section, we look at 
quiver diagrams of the following form:
\begin{equation*}
  0 = V_0\stackrel[\beta_0]{\alpha_0}{\rightleftarrows}
  V_1\stackrel[\beta_1]{\alpha_1}{\rightleftarrows}
  V_2\stackrel[\beta_2]{\alpha_2}{\rightleftarrows}\dots
  \stackrel[\beta_{r-2}]{\alpha_{r-2}}{\rightleftarrows}V_{r-1}
  \stackrel[\beta_{r-1}]{\alpha_{r-1}}{\rightleftarrows} V_r
  = \C^n 
\end{equation*}
where \( V_i \) is a complex vector space of complex dimension \( n_i
\) and \( \alpha_0 = \beta_0 =0 \). The space \( M \) of quivers for fixed
dimension vector \( (n_1, \dots, n_r) \) is a flat hyperk\"ahler vector space.

There is a hyperk\"ahler action of \( \Un(n_1) \times \dots \times
\Un(n_r) \) on this space given by
\begin{equation*}
  \alpha_i \mapsto g_{i+1} \alpha_i g_i^{-1},\quad
  \beta_i \mapsto g_i \beta_i g_{i+1}^{-1} \qquad (i=1,\dots r-1),
\end{equation*}
with \( g_i \in \Un(n_i) \) for \( i=1, \dots, r \). 

Let \( \tilde{H} \) be the subgroup, isomorphic to \( \Un(n_1) \times
\dots \times \Un(n_{r-1}) \), given by setting \( g_r=1 \), and let
\(H = SU(n_1) \times \dots \times SU(n_{r-1}) \leqslant \tilde{H}\).

\begin{definition}
  \label{defQ}
  The \emph{universal hyperk\"ahler implosion for \( \SU(n) \)} will be
  the hyperk\"ahler quotient \( Q = M \hkq H \), where \( M, H \) are
  as above with \( r=n \) and \( n_j =j \), \( (j=1, \dots, n) \), (that
  is, the case of full flag quivers).
\end{definition}
The hyperk\"ahler moment map equations for the \( H \)-action
 are (in the full flag case) 
\begin{align}
  \label{eq:mmcomplex}
  \alpha_i \beta_i - \beta_{i+1}
  \alpha_{i+1} &= \lambda^\C_{i+1} I \qquad (0 \leqslant i \leqslant
  n-2)
\end{align}
where \( \lambda^\C_i \in \C \) for \( 1 \leqslant i \leqslant n-1 \),
and
\begin{equation} \label{eq:mmreal}
  \alpha_i \alpha_i^* - \beta_i^* \beta_i +
  \beta_{i+1} \beta_{i+1}^* - \alpha_{i+1}^* \alpha_{i+1} =
  \lambda^\R_{i+1} I \quad (0 \leqslant i \leqslant n-2), 
\end{equation}
where \( \lambda^\R_i \in \R\) for \( 1 \leqslant i \leqslant n-1 \). 
 Now \( Q \) has a residual action of
\( (S^1)^{n-1} =\tilde{H}/H \) 
as well as an action of \( \SU(n_r) = \SU(n) \). In \S4 we will identify \( (S^1)^{n-1} \) with \( T \), the maximal torus of \( SU(n) \).
There
is also an \( Sp(1)=\SU(2) \) action which is not hyperk\"ahler but rotates the complex structures.
Using the standard theory relating symplectic and GIT quotients,
we have a description of \( Q=M \hkq H \), as the
quotient (in the GIT sense) of the subvariety defined by the complex moment map
equations (\ref{eq:mmcomplex})
by the action of
\begin{gather}
  H_\C = \prod_{i=1}^{n-1}\SL(n_i, \C) \notag \\ 
  \label{SLaction1}
  \alpha_i \mapsto g_{i+1} \alpha_i g_i^{-1}, \quad \beta_i \mapsto
  g_i \beta_i g_{i+1}^{-1} \qquad (i=1,\dots n-2),\\
  \label{SLaction2}
  \alpha_{n-1} \mapsto \alpha_{n-1} g_{n-1}^{-1}, \quad \beta_{n-1}
  \mapsto g_{n-1} \beta_{n-1},
\end{gather}
where \( g_i \in \SL(n_i, \C) \). 

We introduce the element \( X = \alpha_{n-1} \beta_{n-1}
\in \Hom (\C^n, \C^n) \), which is invariant under the action of \(
\prod_{i=1}^{n-1}\GL(n_i,\C) \) and transforms by conjugation under
the residual \( \SL(n,\C)=\SL(n_r,\C) \) action on \( Q \).  We thus have  a \(
T_{\C} \)-invariant and \( \SL(n,\C) \)-equivariant map \( Q
\rightarrow \sln(n,\C) \) given by:
\begin{equation*}
  (\alpha, \beta) \mapsto X - \frac1n \tr(X) I_n
\end{equation*}
where \( I_n \) is the \( n\times n \)-identity matrix. In fact this is
the complex moment map for the residual \( SU(n) \) action on \( Q \).
It is shown in \cite{DKS} that \( X \) satisfies an equation \( X(X+ \nu_1) \dots (X+ \nu_{n-1})=0 \)
where \( \nu_i = \sum_{j=i}^{n-1} \lambda_j \). This generalises the equation \( X^n =0 \) in
 the quiver construction of the nilpotent
variety in \cite{KS}. 

In general it is useful to compare our construction with that in
\cite{KS}.  There one performs a hyperk\"ahler quotient by \( \tilde{H} \),
rather than \( H \), so all \( \lambda_i \) are zero.  In our
situation the \( \lambda_i \) are not constrained to be zero, and in fact give the value of the
complex moment map for the residual \( T \) action on \( Q \).

In \cite{DKS} we first analysed the points of the implosion that give
closed orbits for the \( T_{\C} \) action, or equivalently, quivers that satisfy the
equations (\ref{eq:mmcomplex}) and give closed orbits for the action of \( \tilde{H}_{\C} \) as well as \( H_\C \).
Such quivers can be split into a sum of a quiver with \( \alpha_i \) injective and \( \beta_i \) 
surjective, and a collection of quivers where the non-zero maps are isomorphisms.

In general one must consider quivers that satisfy 
(\ref{eq:mmcomplex}) and  
give a closed orbit for the \( H_{\C} \) action but not necessarily for
the \( \tilde{H}_{\C} \) action. However for each such quiver
we may rotate complex structures so that the closed orbit condition is
actually satisfied for a larger subgroup of \( \tilde{H}_{\C} \). In this
way we obtain a stratification for the implosion.

Using the methods that appeared in the analysis of the symplectic
implosion, we described in \cite{DKS} \S7 the strata for the hyperk\"ahler implosion in
terms of complex-symplectic quotients of \( T^*SL(n,\C) \) by extensions
of abelian groups by commutators of parabolics.  In more detail, we
can follow the argument in the symplectic case to standardise the
surjective maps \( \beta_i \) as \( (0 | I) \). 
The equations (\ref{eq:mmcomplex}) now enable us to  
find \( \alpha_i \) in terms of \( \alpha_{i+1} \) and \( \lambda_{i+1}^{\C} \).
Now knowledge of (the
tracefree part of) \( X =\alpha_{n-1} \beta_{n-1} \), together with the
equations (\ref{eq:mmcomplex}), enables us to
work down the quiver inductively
determining all the \( \alpha_i \). Further details of some of these
arguments are given in \S 5 and \S 6, as well as in \cite{DKS}.

The universal hyperk\"ahler implosion \( Q \) contains an open set which may be identified with
\( SL(n,\C) \times_{N} \bmf \), the complex-symplectic quotient of \( T^*
SL(n,\C) \) by the maximal unipotent \( N \). This arises as the locus of
full flag quivers with all \( \beta_i \) surjective.
The full implosion \( Q \) may in fact be identified with the
non-reductive GIT quotient (\( SL(n, \C) \times \bmf) \symp N \). The
hyperk\"ahler torus quotients of \( Q \) can be identified for any fixed
complex structure with the complex-symplectic reductions of \( Q \) by the
complexified torus \( T_{\C} \) in the sense of GIT. That is, we take the
GIT quotients with respect to \( T_{\C} \) of the level sets of the
complex moment map for the action of \( T \) on \( Q \). These
complex-symplectic reductions give us the Kostant varieties which are
the subsets of \( \sln(n, \C) \) obtained by fixing the eigenvalues. In
particular torus reduction at level \( 0 \) gives the nilpotent
variety. If, by contrast, we take the geometric (rather than GIT)
complex-symplectic reduction at level \( 0 \) of \( SL(n, \C)
\times_{N}\bmf \), we obtain \( (SL(n,\C) \times_{N} \n)/T_{\C} \), which is
the
Springer resolution \( SL(n, \C) \times_{B} \n \) of the nilpotent
variety.

As in the symplectic case, the hyperk\"ahler implosion is usually a
singular stratified space. In fact the symplectic implosion may be
realised as the fixed point set of a circle action on the
hyperk\"ahler implosion, so if the latter is smooth then so is the
former, which implies by results of \cite{GJS} that \( K \) is, up to
covers, a product of copies of \( SU(2) \). If \( K=SU(2) \) the implosion is
just flat \( \HH^2 \).

\section{Hypertoric varieties} 
Classical toric varieties arise
as symplectic quotients of \( \C^d \) by a subtorus of \( (S^1)^d \), and have
a symplectic action of a compact torus \( T \) whose real dimension is half that of the
toric variety \cite{G}, \cite{De}. The image of the toric variety 
under the associated moment map
is called the Delzant polytope, and the toric variety is determined up to \( T \)-equivariant isomorphism by \( T \) and this polytope in \( \lie{t}^* \).

We recall that a \emph{hypertoric} (or \emph{toric hyperk\"ahler})
variety is, by analogy, obtained as a hyperk\"ahler quotient
of flat quaternionic space \( \HH^d \) by a subtorus \( \bf N \) of \( T^d \). If the 
subtorus is of codimension \( n \) in \( T^d \), the associated hypertoric
variety \( \HH^d \hkq {\bf N} \) has real dimension \( 4n \) and has a
hyperk\"ahler action of \( T^n \cong T^d/{\bf N} \). 
The hyperk\"ahler moment map for this action is a surjection onto \( \R^{3n} \)
and much of the geometry of the hypertoric  is encoded in a collection
of codimension 3 affine subspaces (the \emph{flats}) in \( \R^{3n} \).
These play in some respects a role analogous to that of the hyperplanes
giving the faces of the Delzant polytope for classical toric varieties.
In particular the fibre of the moment map over a point
in \( \R^{3n} \) is a torus determined by the collection of flats
passing through that point.
We refer the reader to \cite{BD}, \cite{HS} for further background
on hypertorics.

We want to relate the hyperk\"ahler implosion \( Q \) to
the hypertoric variety associated to the arrangement of flats induced by the
hyperplane arrangement given
by the root planes in the Lie algebra \( \lie{t} \) of the maximal torus \( T \) of \( K=\SU(n) \).

\begin{definition} \label{defhypertoric}
Let \( M_T \) be the subset of \( M \) consisting of all hyperk\"ahler
quivers of the form 
\begin{equation*} 
\alpha_k = \left( \begin{array}{ccccc}
\nu_1^k & 0 & 0 & \cdots & 0\\
0 & \nu_2^k & 0 & \cdots & 0\\
 & & \cdots & & \\
0 & \cdots & 0 & 0 & \nu_k^k\\
0 & \cdots & 0 & 0 & 0 \end{array} \right) \end{equation*}
and
\begin{equation*}
\beta_k = \left( \begin{array}{cccccc}
\mu_1^k & 0 & 0 & 0 & \cdots & 0\\
0 & \mu_2^k & 0 & 0 & \cdots & 0\\
 & & \cdots & & &  \\
0 & \cdots & 0 & 0 & \mu_k^k & 0  \end{array} \right)
\end{equation*}
for some \( \nu^k_i, \mu_i^k \in \C \).
Recall from \cite{DKS} that we use \( M^{\hks} \) to denote the set of
\emph{hyperk\"ahler stable} quivers, that is, those that
after an appropriate rotation of complex structures
have all \( \alpha_i \) injective and all \( \beta_i \) surjective.
Let \( M_T^{\hks} = M^{\hks} \cap M_T \) be the subset of \( M^{\hks} \) consisting of all hyperk\"ahler
quivers of the form above; thus \( M_T^{\hks} \) consists of all quivers of the form
above such that \( \mu_i^k \) and \( \nu_i^k \) are not simultaneously zero
for any pair \( (i,k) \) with \( 1 \leqslant i \leqslant k < n \).
 
 Note that each of the
compositions \( \alpha_k \beta_k \), \( \beta_k\alpha_k \), \( \alpha_k \alpha_{k}^* \),
\( \alpha_{k}^*\alpha_k \), \( \beta_k\beta_{k}^* \) and \( \beta_{k}^*\beta_k \) is a 
diagonal matrix, so that for quivers of this form the hyperk\"ahler moment map
equations for the action of \( H = \prod_{k=1}^{n-1} \SU(k) \) reduce to 
the hyperk\"ahler moment map
equations for the action of its maximal torus
\begin{equation*}
T_H = \prod_{k=1}^{n-1} T_k
\end{equation*}
where \( T_k \) is the standard maximal torus in \( \SU(k) \). Moreover two
hyperk\"ahler stable quivers of this
form satisfying the hyperk\"ahler moment map
equations lie in the same orbit for the  action of \( H \) 
if and only if they lie in the same orbit for the action of its maximal torus
\( T_H \). Thus we get a natural map
\begin{equation*}
\iota:M_T \hkq T_H \to Q = M\hkq H
\end{equation*}
which restricts to an embedding
\begin{equation*}
\iota: Q^{\hks}_T \to Q
\end{equation*}
where \( Q^{\hks}_T = M_T^{\hks} \hkq T_H \).
\end{definition}

\begin{remark} \label{remhypertoric}
Note that \( M_T = \bigoplus_{k=1}^{n-1} \mathbb{H}^k = \mathbb{H}^{n(n-1)/2} \) is
a (flat) hypertoric variety with respect to the action of the standard
maximal torus \( T_{\tilde{H}} = (S^1)^{n(n-1)/2} \) of \( \tilde{H} = \prod_{k=1}^{n-1} \Un(k) \). The associated arrangement of flats in \( \R^{3n(n-1)/2} = \R^3 \otimes
\R^{n(n-1)/2} \) is just that induced by
the hyperplane arrangement given by the 
coordinate hyperplanes in \( \lie{t}_{\tilde{H}} = \R^{n(n-1)/2} \). Thus
\( M_T \hkq T_H \)
is a hypertoric variety for the induced action of 
\begin{equation*}
T_{\tilde{H}}/T_H = \prod_{k=1}^{n-1} \Un(k)/\SU(k) = (S^1)^{n-1}.
\end{equation*}
Moreover
we can identify \( T_{\tilde{H}}/T_H \) with the standard maximal torus \( T \) of 
\( K=\SU(n) \) in such a way that the induced action of \( T_{\tilde{H}}/T_H \)
on \(  Q^{\hks}_T \) coincides with the restriction to \( T \) of the action of \( K \) on
\(  Q^{\hks}_T \) embedded in \( Q=M \hkq H \) as above, since 
\begin{gather*}
  \left( \begin{array}{cccccc}
t_1^k & 0 & 0 & \cdots & 0 & 0\\
0 & t_2^k & 0 & \cdots & 0 & 0\\
 & & \cdots & & &\\
0 & \cdots & 0 & 0 & t_{k}^{k} & 0 \\
0 & \cdots & 0 & 0 & 0 & t_{k+1}^k \end{array} \right)
 \left( \begin{array}{ccccc}
\nu_1^k & 0 & 0 & \cdots & 0\\
0 & \nu_2^k & 0 & \cdots & 0\\
 & & \cdots & & \\
0 & \cdots & 0 & 0 & \nu_k^k\\
0 & \cdots & 0 & 0 & 0 \end{array} \right)\\
 =
 \left( \begin{array}{ccccc}
\nu_1^k & 0 & 0 & \cdots & 0\\
0 & \nu_2^k & 0 & \cdots & 0\\
 & & \cdots & & \\
0 & \cdots & 0 & 0 & \nu_k^k\\
0 & \cdots & 0 & 0 & 0 \end{array} \right)
 \left( \begin{array}{cccccc}
t_1^k & 0 & 0 & \cdots & 0 \\
0 & t_2^k & 0 & \cdots & 0 \\
 & & \cdots & & \\
0 & \cdots & 0 & 0 & t_{k}^k \end{array} \right)
\end{gather*}
and 
\begin{gather*}
  \left( \begin{array}{ccccc}
t_1^k & 0 & 0 & \cdots & 0\\
0 & t_2^k & 0 & \cdots & 0\\
 & & \cdots & & \\
0 & \cdots & 0 & 0 & t_k^k \end{array} \right)
\left( \begin{array}{cccccc}
\mu_1^k & 0 & 0 & 0 & \cdots & 0\\
0 & \mu_2^k & 0 & 0 & \cdots & 0\\
 & & \cdots & & &  \\
0 & \cdots & 0 & 0 & \mu_k^k & 0  \end{array} \right)
\\ =
\left( \begin{array}{cccccc}
\mu_1^k & 0 & 0 & 0 & \cdots & 0\\
0 & \mu_2^k & 0 & 0 & \cdots & 0\\
 & & \cdots & & &  \\
0 & \cdots & 0 & 0 & \mu_k^k & 0  
\end{array} \right)
  \left( \begin{array}{cccccc}
t_1^k & 0 & 0 & \cdots & 0 & 0\\
0 & t_2^k & 0 & \cdots & 0 & 0\\
 & & \cdots & & &\\
0 & \cdots & 0 & 0 & t_k^k & 0\\
0 & \cdots & 0 & 0 & 0 & t_{k+1}^k \end{array} \right)
\end{gather*}
for any \( \nu^k_i, \mu^k_i \) and \( t^k_i \) in \( \C \). Here if \( (s_1,\ldots,s_{n-1}) \)
are the standard coordinates on the Lie algebra \( \R^{n-1} \) of \( (S^1)^{n-1} \) and
\( (\tau_1, \ldots, \tau_n) \) are the standard coordinates of the Lie algebra
\( \R^n \) of the maximal torus of \( \Un(n) \) consisting of the diagonal matrices, then
we identify \( \R^{n-1} \) with the subspace of \( \R^n \) defined by \( \tau_1 + \cdots + \tau_n = 0 \) via the relationship \( s_j = \tau_{j+1} + \cdots + \tau_n \) for \( 1 \leq j \leq n-1 \).
With respect to this identification, \( M_T \hkq T_H \) becomes the hypertoric
variety for \( T \) associated to the hyperplane arrangement in its Lie algebra
\( \lie{t} \) given by the root planes.
\end{remark}

\section{Stratifying the universal hyperk\"{a}hler 
implosion into  hyperk\"{a}hler strata}
The universal hyperk\"{a}hler implosion \( Q = M \hkq H \) for \( K = \SU(n) \)
is a singular space with a stratification into locally closed hyperk\"{a}hler submanifolds \( Q_{(S,\delta)} \) (cf.\,  \cite{DKS} Theorem 6.15). These strata
\( Q_{(S,\delta)} \) can be indexed by subsets
\begin{equation*}
S = \{(i_1,j_1), (i_2,j_2), \ldots, (i_p,j_p)\}
\end{equation*}
of \( \{1,\ldots,n\} \times \{1,\ldots,n\} \) with   \( i_1, \dots, i_p \) distinct and \(
  j_1 < j_2 < \dots < j_p \), and sequences \( \delta=(d_1,\dots,d_p) \) of strictly positive integers such that if \( 1 \leqslant k
  \leqslant n \) then
  \begin{equation*}
    m_k = k - \sum_{\substack{h:\\ 1 \leqslant h \leqslant p\\ i_h
    \leqslant k < j_h}} d_h 
  \end{equation*}
  satisfies \( 0 = m_0 \leqslant m_1 \leqslant \dots \leqslant m_n = n
  \). The open stratum
  \( Q_{(\emptyset,\emptyset)}= Q^{\hks} \), which is indexed by the
empty set \( S = \emptyset \) and the empty sequence \( \delta=\emptyset \),
consists of those elements of \( Q = M \hkq H \) represented by 
hyperk\"ahler stable quivers.

More generally, for any \( S \) and \( \delta \) as above, the stratum 
\( Q_{(S,\delta)} \) is the image of a hyperk\"ahler embedding into \( Q \) of
a hyperk\"ahler modification 
\( \hat{Q}_1^{\hks} \)
(in the sense of Definition \ref{hkmodification} below, following
\cite{DSmod}) of the open subset \( Q_1^{\hks} \) represented by hyperk\"ahler quivers in the hyperk\"ahler quotient
\begin{equation*}
Q_1 = M_1 \hkq H_S
\end{equation*}
where  \( M_1 \) is the space of quivers of the form
  \begin{equation}
    \label{quiverdagger}
    0 \stackrel[\beta_0]{\alpha_0}{\rightleftarrows}
    \C^{m_1}\stackrel[\beta_1]{\alpha_1}{\rightleftarrows}
    \C^{m_2}\stackrel[\beta_2]{\alpha_2}{\rightleftarrows}\dots
    \stackrel[\beta_{n-2}]{\alpha_{n-2}}{\rightleftarrows}
    \C^{m_{n-1}}
    \stackrel[\beta_{n-1}]{\alpha_{n-1}}{\rightleftarrows}
    \C^{m_n} = \C^n
  \end{equation}
and \( H_S \) is the subgroup of
\begin{equation*}
\prod_{k=1}^{n-1} \Un(m_k)  \leqslant \tilde{H} = 
\prod_{k=1}^{n-1} \Un(k)
\end{equation*}
defined as follows.
 
\begin{definition}
  \label{defHs}
  To any set \( S \) of pairs \( (i,j) \) with \( i,j \in
  \{1,\dots,n\} \) we can associate a subtorus \( T_S \) of \(
  {T} = (S^1)^{n-1} \) such that the Lie algebra of \( T_S \) is
  generated by the vectors \( e_{ij}= (0, \dots,0, 1, 1,\dots,1, 0,
  \dots,0) \), which have \( 1 \) in places \( i,\dots, j-1 \) and
  zero elsewhere, where \( i,j \) range over all pairs \( (i,j) \in S
  \) with \( i < j \).
  Now consider the short
  exact sequence
  \begin{equation*}
    1 \rightarrow H=\prod_{k=1}^{n-1} \SU(m_k) \rightarrow 
    \prod_{k=1}^{n-1} \Un(m_k) 
    \stackrel{\phi}{\rightarrow} T \rightarrow 1, 
  \end{equation*}
  where \( \phi \) is the obvious product of determinant maps, and
  define \( H_S \) to be the preimage
  \begin{equation*}
    H_S = \phi^{-1} (T_S)
  \end{equation*}
  of \( T_S \) in 
 \(  
\prod_{k=1}^{n-1} \Un(m_k)  \).
\end{definition}

\begin{definition} \label{hkmodification}
 Motivated by the concept of hyperk\"ahler modification introduced in \cite{DSmod}, we define \( \hat{Q}_1^{\hks} \) as
\begin{equation*}
\hat{Q}_1^{\hks} = (Q_1^{\hks} \times (\mathbb{H} \setminus \{ 0\}) ^\ell) \hkq (S^1)^\ell.
\end{equation*}
Here  \( \ell = |L| \) is the size of the set 
\begin{equation*}
L= \{(h,k): 1 \leqslant h \leqslant p, \,\, i_h \leqslant k < j_h -1 \}.
\end{equation*}
The action of \( (S^1)^\ell \) on \( (\mathbb{H} \setminus \{ 0\}) ^\ell \)
is the standard one, while the 
action of \( (S^1)^\ell \) on  \( Q_1 \) is given by the homomorphism
\begin{equation*}
(S^1)^\ell \to {T} = (S^1)^{n-1}
\end{equation*}
whose restriction to the copy of \( S^1 \) in \( (S^1)^\ell \) labelled by \( (h,k) \in L \)
sends the standard generator of the Lie algebra of \( S^1 \) to the vector
\begin{equation*}
e_{k+1\,\, j_h} = (0, \ldots, 0,1,1,\ldots,1,0,\ldots,0)
\end{equation*}
in the Lie algebra of \( {T} = (S^1)^{n-1} \) which has 1 in places
\( k+1 \) to \( j_h - 1 \) and 0 elsewhere. 
\end{definition}
The stratum \( Q_{(S,\delta)} \) is the image of a hyperk\"ahler embedding
  \begin{equation*}
    \hat{Q}_1^{\hks} \to Q
  \end{equation*}
  which is \( \SU(2) \)-equivariant and is defined as follows.
 Consider a quiver
\eqref{quiverdagger}
  together with an element \( (\gamma_k^{(h)}) \) of \( \mathbb{H}^\ell  \)
such that  \( \gamma_k^{(h)} = \alpha_k^{(h)} + j \beta_k^{(h)}  \)
for \(1 \leqslant h \leqslant p \) and  \( i_h \leqslant k < j_h
  - 1 \),
 satisfying the \( H_{S} \)-hyperk\"ahler moment map equations
  \begin{gather*}
    \alpha_i \beta_i - \beta_{i+1}
    \alpha_{i+1} = \lambda^\C_{i+1} I, \\
    \alpha_i \alpha_i^* - \beta_i^* \beta_i +
    \beta_{i+1} \beta_{i+1}^* - \alpha_{i+1}^* \alpha_{i+1} =
    \lambda^\R_{i+1} I,  
  \end{gather*} 
 for \(1 \leqslant i \leqslant n-1 \), where \( \sum_{k = i_h}^{j_h-1} \lambda_k^\C = 0 \) and \( \sum_{k =
  i_h}^{j_h-1} \lambda_k^\R=0 \) for \( 1 \leqslant h \leqslant q \),
and the \( (S^1)^\ell \)-hyperk\"ahler moment map equations
\begin{gather*}
\alpha_k^{(h)}\beta_k^{(h)} = \lambda^\C_{k+1} + \lambda^\C_{k+2}
   + \dots + \lambda^\C_{j_h -1} \quad\text{and}\\
   \abs{\alpha_k^{(h)}}^2 - \abs{\beta_k^{(h)}}^2 = \lambda^\R_{k+1} +
   \lambda^\R_{k+2} + \dots + \lambda^\R_{j_h -1} 
 \end{gather*}
for \(1 \leqslant h \leqslant p \) and  \( i_h \leqslant k < j_h
  - 1 \). Our embedding takes the \( H_{S} \times (S^1)^\ell \)-orbit of
this configuration  to the \( H \)-orbit of the quiver
  \begin{equation}
    \label{quiverdoubledagger}
    0 \stackrel[\tilde{\beta}_0]{\tilde{\alpha}_0}{\rightleftarrows}
    \dots \rightleftarrows 
    \C^{m_k} \oplus \bigoplus_{h :\, i_h \leqslant k < j_h} \C^{d_h} 
    \rightleftarrows 
    \dots
    \stackrel[\tilde{\beta}_{r-1}]{\tilde{\alpha}_{r-1}}{\rightleftarrows} 
    \C^n
  \end{equation}
  which is the orthogonal direct sum of \eqref{quiverdagger} with the quivers given
for \( 1
  \leqslant h \leqslant p \) by
  \begin{equation*}
    \C^{d_h}
    \stackrel[\beta_{i_h}^{(h)}]{\alpha_{i_h}^{(h)}}{\rightleftarrows}
    \C^{d_h} \rightleftarrows \dots \rightleftarrows \C^{d_h}
    \stackrel[\beta_{j_h-2}^{(h)}]{\alpha_{j_h-2}^{(h)}}{\rightleftarrows}
    \C^{d_h}    
  \end{equation*}
  in the places \( i_h,i_h + 1, \dots ,j_h -1 \).  Here the maps \(
  \alpha_k^{(h)} \), \( \beta_k^{(h)} \), for \( i_h \leqslant k < j_h
  - 1 \), are multiplication by the complex scalars, also denoted by
  \( \alpha_k^{(h)} \), \( \beta_k^{(h)} \), that satisfy
 \( \alpha_k^{(h)} + j \beta_k^{(h)} = \gamma_k^{(h)} \).

\begin{remark} \label{remhypertoricstrat}
Note that the stabiliser in \( T=T_{\tilde{H}}/T_H \) of 
any \( \q \in Q_{(S,\delta)} \) is  the subtorus \( T_S \) of \( T \) defined in Definition \ref{defHs}, which is the product \( (S^1)^p \) of \( p \) copies of \( S^1 \)
where the \( j \)th copy of \( S^1 \) acts by scalar multiplication on the summand
 \begin{equation*}
    \C^{d_h}
    \stackrel[\beta_{i_h}^{(h)}]{\alpha_{i_h}^{(h)}}{\rightleftarrows}
    \C^{d_h} \rightleftarrows \dots \rightleftarrows \C^{d_h}
    \stackrel[\beta_{j_h-2}^{(h)}]{\alpha_{j_h-2}^{(h)}}{\rightleftarrows}
    \C^{d_h}    
  \end{equation*}
of the quiver \( \q \).
\end{remark}

\section{A refined stratification of the universal hyperk\"ahler implosion}
In the last section we recalled the stratification of the universal hyperk\"ahler implosion \( Q= M\hkq H \) for \( K=\SU(n) \) into hyperk\"ahler strata  
\( Q_{(S,\delta)} \).
 In this section we will refine this stratification to obtain strata which
are not in general hyperk\"ahler but which reflect the structure of the group \( K=\SU(n) \); in particular we would like to find a description of the
universal hyperk\"ahler implosion which permits generalisation to other compact groups.
First let us consider the hyperk\"ahler moment map 
\begin{equation*}
\mu_{(S^1)^{n-1}}: Q \to (\R^3)^{n-1}
\end{equation*}
for the induced action of 
\begin{equation*}
T = (S^1)^{n-1} = \prod_{k=1}^{n-1} \Un(k)/\SU(k) = \tilde{H}/H
\end{equation*}
on \( Q = M \hkq H \). We are abusing notation slightly here by using the same symbol \( T \) to denote both \( (S^1)^{n-1} \) and the standard maximal torus of \( K=\SU(n) \). These tori are of course isomorphic; we will always make the particular choice of identification given in Remark \ref{remhypertoric}, so
that the restriction to \( Q_T^{\hks} \) of the action of \( T \) as a subgroup of
\( K \) agrees with the restriction of the action of \( T \) identified with \( (S^1)^{n-1} \). 
This  hyperk\"ahler moment map takes a quiver 
 which satisfies the  equations
(\ref{eq:mmcomplex}),(\ref{eq:mmreal})
to the element of 
\begin{equation*}
\lie{t} \otimes \R^3 = (\R^3)^{n-1} = (\C \oplus \R)^{n-1}
\end{equation*}
given by
\begin{equation*}
(\lambda_1, \ldots, \lambda_{n-1})= (\lambda_1^\C,\lambda_1^\R, \ldots, \lambda_{n-1}^\C,\lambda_{n-1}^\R).
\end{equation*}
We will define a stratification of \( (\R^3)^{n-1} \) which we can pull back via the restriction of \( \mu_{(S^1)^{n-1}} \) to each hyperk\"ahler stratum \( Q_{(S,\delta)} \)
of \( Q \).
\begin{definition} If \( (\lambda_1, \ldots, \lambda_{n-1}) \in (\R^3)^{n-1} \) there is an associated equivalence relation \( \sim \) on \( \{1,\dots,n\} \)
such that if \( 1 \leqslant i<j \leqslant n \) then
\begin{equation*}
  i \sim j \iff \sum_{k=i}^{j-1} \lambda_k = 0
  \ \text{in}\ \R^3 .
\end{equation*}
There is thus a stratification of \( (\R^3)^{n-1}=\lie{t} \otimes \R^3 \) into strata
\( (\R^3)^{n-1}_{\sim} = (\lie{t} \otimes \R^3)_\sim \),
indexed by the set of equivalence relations  \( \sim \) on \( \{1,\dots,n\} \),
where
\begin{equation*}
(\R^3)^{n-1}_{\sim} = \{ (\lambda_1, \ldots, \lambda_{n-1}) \in (\R^3)^{n-1}:
\mbox{ if } 1 \leqslant i<j \leqslant n \mbox{ then }
\end{equation*}
\begin{equation*}
i \sim j \iff \sum_{k=i}^{j-1} \lambda_k = 0
  \ \text{in}\ \R^3\}.
\end{equation*}
\end{definition}
\begin{remark} \label{remksim}
Under the identification of \( T \) with \( (S^1)^{n-1} \) given in 
Remark~\ref{remhypertoric} this stratification of \( (\R^3)^{n-1}= \lie{t} \otimes \R^3 \)
is the tensor product with \( \R^3 \) of the stratification of \( \lie{t} \) associated to
the hyperplane arrangement given by the root planes in \( \lie{t} \).
Note also that an equivalence relation \( \sim \) on   \( \{1, \ldots, n\} \) determines
and is determined by a subgroup \( K_\sim \) of \( K=\SU(n) \), where \( K_\sim \) is the stabiliser in \( K \) of any \(  (\lambda_1, \ldots, \lambda_{n-1}) \in \lie{t} \otimes \R^3 \) which lies in the stratum \( (\R^3)^{n-1}_{\sim} \) of \( (\R^3)^{n-1} \)
identified with \(  \lie{t} \otimes \R^3 \) as in Remark \ref{remhypertoric}.
\end{remark}
Observe that each stratum \( (\R^3)^{n-1}_{\sim} \) is an open subset of a linear subspace of the real vector space \( (\R^3)^{n-1} \).
\begin{definition} \label{QSdeltasim}
Given a hyperk\"ahler stratum \( Q_{(S,\delta)} \) of \( Q \) as in \S4, together with an equivalence relation \( \sim \) on \( \{1, \ldots, n\} \), define
\begin{equation*}
Q_{(S,\delta,\sim)} = Q_{(S,\delta)} \cap \mu_{(S^1)^{n-1}}^{-1}((\R^3)^{n-1}_{\sim}),
\end{equation*}
that is,  the inverse image of the stratum \( (\R^3)^{n-1}_{\sim} \) in \( (\R^3)^{n-1} \)
under the restriction to \( Q_{(S,\delta)} \) of  the hyperk\"ahler moment map 
\( \mu_{(S^1)^{n-1}}: Q \to (\R^3)^{n-1} \).
\end{definition}
\begin{remark} 
\label{lem5.14} 
We recall from \cite{DKS} that, given  a quiver 
 which satisfies the complex moment map equations
(\ref{eq:mmcomplex}),
 we may decompose
each space in the quiver into
generalised eigenspaces \( \ker (\alpha_i \beta_i - \tau I)^m \) of 
\( \alpha_i \beta_i \).
We showed that \( \beta_i \) restricts
to a map
\begin{equation}
  \label{betai}
  \beta_i \colon \ker (\alpha_i \beta_i - \tau I)^m \rightarrow
  \ker(\alpha_{i-1} \beta_{i-1} - (\lambda_i^\C + \tau)I)^m. 
\end{equation}
Similarly \( \alpha_i \) restricts to a map
\begin{equation}
  \label{alphai}
  \alpha_i \colon \ker (\alpha_{i-1} \beta_{i-1} - (\lambda_i^\C +
  \tau) I)^m \rightarrow \ker (\alpha_i \beta_i - \tau I)^m. 
\end{equation}
Moreover we showed the maps \eqref{betai} and \eqref{alphai} are bijective
unless \( \tau = 0 \).
It follows that \( \tau \neq 0 \) is an eigenvalue of \( \alpha_i
\beta_i \) if and only if \( \tau + \lambda_i^\C \neq \lambda_i^\C \)
is an eigenvalue of \( \alpha_{i-1} \beta_{i-1} \). Moreover \( \alpha_i \beta_i 
\) has zero as an eigenvalue and \( \alpha_i, \beta_i \) restrict to
maps between the associated generalised eigenspace with eigenvalue \(
0 \) and the generalised eigenspace for \( \alpha_{i-1} \beta_{i-1} \)
associated to \( \lambda_i^\C \).  One can deduce
 (cf. Lemma 5.14 of \cite{DKS}) that
  \begin{equation*}
    \alpha_{n-1} \beta_{n-1} - \frac1n \tr(\alpha_{n-1} \beta_{n-1}) I_n \in \sln(n, \C)
  \end{equation*}
now has eigenvalues \( \kappa_1,\dots,\kappa_n \), where
  \begin{equation*}
    \begin{split}
      \kappa_j &= \frac1n\Bigl( \lambda^\C_1 + 2 \lambda_2^\C + \dots +
      ({j-1}) \lambda^\C_{j-1} \eqbreak[4] -(n-j) \lambda_j^\C -
      (n-{j-1})\lambda_{j+1}^\C - \dots - \lambda_{n-1}^\C\Bigr).
    \end{split}
  \end{equation*}
  In particular if \( i<j \) then
  \begin{equation} \label{eqnkappa}
    \kappa_j - \kappa_i = \lambda_i^\C + \lambda_{i+1}^\C + \dots +
    \lambda_{j-1}^\C . 
  \end{equation}
We deduce that if \( i \sim j \) then we have equality of the eigenvalues
\( \kappa_i \) and \( \kappa_j \).
\end{remark}

We would like to find an indexing set for the subsets \(  Q_{(S,\delta, \sim)}  \)
which reflects the group theoretic structure of \( K \). As we observed in
Remark \ref{remksim} the choice of \( \sim \) corresponds to the choice of a
subgroup \( K_\sim \) of \( K \) which is the 
compact real form of a Levi subgroup of \( K_\C \);
 this subgroup \( K_\sim \) is the centraliser of \( \mu_{(S^1)^{n-1}}(\q) \in \lie{t} \otimes \R^3 \) for
any \( \q \in  Q_{(S,\delta, \sim)}  \).
Our next aim is to show that
once \( \sim \) or equivalently \( K_\sim \) is chosen, the choice of \( (S,\delta) \) 
corresponds to the choice of a nilpotent adjoint orbit \( \mathcal{O} \subseteq (\lie{k}_\sim)_\C \). We will see that  if \( \q \in  Q_{(S,\delta, \sim)}  \) then,
for a generic choice of complex structure, when we 
decompose \( \C^n \) as the direct sum of the generalised eigenspaces of
\begin{equation*}
\alpha_{n-1} \beta_{n-1} - \frac{1}{n}{\rm tr} (\alpha_{n-1} \beta_{n-1}) I_n \in \lie{k}_\C
\end{equation*}
then the subgroup of \( K_\C \) preserving this decomposition is conjugate to
\( (K_\sim)_\C \). In fact there is some \( g \in K_\C \) such that this subgroup is
\( g(K_\sim)_\C)g^{-1} \) and when we write the
element  
\begin{equation*}
\alpha_{n-1} \beta_{n-1} - \frac{1}{n}{\rm tr} (\alpha_{n-1} \beta_{n-1}) I_n
\end{equation*}
of \( \lie{k}_\C \) as the sum of commuting nilpotent and semisimple elements of \( \lie{k} \),
the semisimple element is the conjugate by \( g \) of
\( \mu_{(S^1)^{n-1}}^\C(q) \)
and the nilpotent element lies in the conjugate by \( g \) of the \( (K_\sim)_\C \)-orbit \( \mathcal{O} \). 
To see this, we
need to recall from \cite{DKS} more about the hyperk\"ahler strata \( Q_{(S,\delta)} \). So suppose that a quiver 
 satisfies the hyperk\"ahler moment map equations
(\ref{eq:mmcomplex}) and (\ref{eq:mmreal}),
and lies in the subset \( Q_{(S,\delta,\sim)} \),
so that it lies in \( Q_{(S,\delta)} \) and 
\begin{equation*}
  i \sim j \iff \sum_{k=i}^{j-1} \lambda_k = 0
  \ \text{in}\ \R^3 .
\end{equation*}
Notice that \( Q_{(S,\delta,\sim)} \) is preserved by the rotation action of \( \SU(2) \)
on \( \R^3 \), and that given 
\begin{equation*}
(\lambda_1, \ldots, \lambda_{n-1})= (\lambda_1^\C,\lambda_1^\R, \ldots, \lambda_{n-1}^\C,\lambda_{n-1}^\R) \in (\R^3)^{n-1},
\end{equation*}
by applying a generic element of \( \SU(2) \) to rotate the complex structures and hence the decomposition \( \R^3 = \C \oplus \R \), we can assume that  if
\( 1 \leqslant i < j \leqslant n \) then   
\begin{equation} \label{reg}
  \sum_{k=i}^{j-1} \lambda_k = 0
  \ \text{in}\ \R^3 \iff \sum_{k=i}^{j-1} \lambda_k^\C = 0
  \ \text{in}\ \C.
\end{equation}
Thus 
\begin{equation*}
Q_{(S,\delta,\sim)} = \SU(2) Q_{(S,\delta,\sim)}^\circ
\end{equation*}
where \( Q_{(S,\delta,\sim)}^\circ \) is the open subset of \( Q_{(S,\delta,\sim)} \) represented by those quivers 
 in \( Q_{(S,\delta,\sim)} \) 
such that
\begin{equation*}
  i \sim j \iff \sum_{k=i}^{j-1} \lambda_k = 0
  \ \text{in}\ \R^3 \iff \sum_{k=i}^{j-1} \lambda_k^\C = 0
  \ \text{in}\ \C.
\end{equation*}
\begin{remark} \label{remkappa}
It now follows immediately from (\ref{eqnkappa}) that for a quiver \( \q \) in 
\( Q_{(S,\delta,\sim)}^\circ \)  we have equality of the eigenvalues
\( \kappa_i \) and \( \kappa_j \) of 
\begin{equation*}
    \alpha_{n-1} \beta_{n-1} - \frac1n \tr(\alpha_{n-1} \beta_{n-1}) I_n \in \sln(n, \C)
  \end{equation*}
 if and only if  \( i \sim j \) . In particular if  \( \q\in 
Q_{(S,\delta,\sim)}^\circ \) then \( (K_\sim)_\C \) is the subgroup of \( K_\C \) which 
preserves the decomposition of \( \q \) into the subquivers determined by the generalised eigenspaces of the compositions \( \alpha_i\beta_i \) as in Remark~\ref{lem5.14} above.
\end{remark}
Let us suppose now that our quiver 
\( \q \) lies in
\( Q_{(S,\delta,\sim)}^\circ \). As in Remark~\ref{lem5.14} 
we can decompose it into  a direct sum of subquivers
  \begin{equation*} \cdots 
    V_i^j
    \stackrel[\beta_{i,j}]{\alpha_{i,j}}{\rightleftarrows}
    V_{i+1}^j \cdots 
  \end{equation*}
determined by the generalised eigenspaces (with eigenvalues \( \tau_{i+1, j} \))
of the compositions \( \alpha_i\beta_i \), such that
\begin{equation*}
\alpha_{i,j} \beta_{i,j} - \beta_{i+1,j} \alpha_{i+1,
  j} = \lambda_{i+1}^\C
\end{equation*}
 and \( \alpha_{i,j}
  \) and \( \beta_{i,j} \) are isomorphisms unless \( \tau_{i+1, j}=0
  \).
  If for some \( j \) we have that \( \alpha_{k,j}, \beta_{k,j} \) are
  isomorphisms for \( i+1 \leqslant k < s \) but not for \( k=i,s \), then
  it follows  that \( \tau_{i+1,j} =
  \tau_{s+1,j}=0 \), hence \( \sum_{k=i+1}^s \lambda_k^\C =0 \), and
so since the quiver lies in \( Q_{(S,\delta,\sim)}^\circ \) we have
\begin{equation*}
\sum_{k=i+1}^s \lambda_k =0 \in \R^3.
\end{equation*}
This means that the quiver satisfies the hyperk\"ahler moment map equations not just for \( H = \prod_{k=1}^{n-1} \SU(k) \)
but for its extension
\begin{equation*}
H_{\{(i,s)\} }
\end{equation*}
by \( S^1 \) in the sense of Definition \ref{defHs}. In particular it satisfies the real moment map equations for this subgroup of \( \tilde{H} = \prod_{k=1}^{n-1} \Un(k) \), and thus its orbit under the complexification \( (H_{\{(i,s)\} })_\C \)
of \( H_{\{(i,s)\} } \) is closed.
\begin{remark}
\label{remdim}
Note that  \(
  \alpha_{i,j} \) and \( \beta_{i,j} \) are isomorphisms  when \( \tau_{i+1, j} \) is non-zero, and hence for each \( i \) there is exactly one \( j \) such that \( \tau_{i+1, j} = 0 \) and then
\begin{equation*}
\dim V_{i+1}^j = 1 + \dim V_{i}^j.
\end{equation*}
\end{remark}
  In the case when \( \tau_{i+1, j} \) is non-zero, and hence \(
  \alpha_{i,j} \) and \( \beta_{i,j} \) are isomorphisms, we may
   contract the subquiver 
 \begin{equation*} \cdots 
    V_i^j
    \stackrel[\beta_{i,j}]{\alpha_{i,j}}{\rightleftarrows}
    V_{i+1}^j \cdots 
  \end{equation*}
by replacing
  \begin{equation*}
    V_{i-1}^j
    \stackrel[\beta_{i-1,j}]{\alpha_{i-1,j}}{\rightleftarrows}
    V_i^j
    \stackrel[\beta_{i,j}]{\alpha_{i,j}}{\rightleftarrows}
    V_{i+1}^j  
    \stackrel[\beta_{i+1,j}]{\alpha_{i+1,j}}{\rightleftarrows}
    V_{i+2}^j
  \end{equation*}
  with
  \begin{equation*}
    V_{i-1}^j
    \stackrel[\beta_{i-1,j}]{\alpha_{i-1,j}}{\rightleftarrows}
    V_i^j
    \stackrel[(\alpha_{i,j})^{-1}\beta_{i+1,j}]{\alpha_{i+1,j}  \alpha_{i,j}}{\rightleftarrows}
        V_{i+2}^j,
  \end{equation*}
  and then the complex moment map equations are satisfied with
  \begin{equation*}
    \alpha_{i-1,j}\beta_{i-1,j} - (\alpha_{i,j})^{-1}\beta_{i+1,j} \alpha_{i+1,j}  \alpha_{i,j} = \lambda_{i-1}^\C + \lambda_i^\C.
  \end{equation*}
   Moreover if we choose an identification of \( V_{i+1}^j \) with \( V_i^j \)
  and apply the action of \( \SL(V_{i,j}) \) to set \( \alpha_{i,j} \)
  to be a non-zero scalar multiple \( aI \) of the identity, then \(
  \beta_{i,j} \) is determined by \( \alpha_{i-1,j}, \alpha_{i+1,j},
  \beta_{i-1,j}, \beta_{i+1,j} \) via the equations
  \eqref{eq:mmcomplex} once we know the scalars \( a \) and \(
  \lambda_i^\C \). We refer the reader to \cite{DKS} for more details
on contraction.
After performing such contractions whenever \( \tau_{i+1, j} \) is non-zero,
we obtain contracted quivers 
  \begin{equation*} \cdots \,\,\,\,
    V_{i}^j
    \stackrel[\alpha_{i,j}^{-1} \cdots \alpha_{s-2,j}^{-1} \beta_{s-1,j}]{\alpha_{s-1,j}\cdots \alpha_{i,j}}{\rightleftarrows}
    V_s^j \,\,\,\, \cdots 
  \end{equation*}
where
\begin{equation*}
V^j_i \cong V^j_{i+1} \cong \cdots \cong V^j_{s-1}
\end{equation*}
and \( \dim V^j_{s-1} = \dim V^j_s - 1 \). Moreover each of these contracted quivers
satisfies the complex moment map equations for the induced action of 
\begin{equation*}
\prod_{i: \dim V^j_{i-1} < \dim V^j_i} \GL(V^j_i)
\end{equation*}
and its orbit under the action of this complex group is closed. 
It then follows from~\cite{KS}
Theorem 2.1 (cf. \cite{DKS} Proposition 5.16) that each contracted subquiver is the direct sum of a quiver of the form 
  \begin{equation*}
    0=V_0^{(*)}
    \stackrel[\beta^{(*)}_0]{\alpha^{(*)}_0}{\rightleftarrows}
    V^{(*)}_1
    \stackrel[\beta^{(*)}_1]{\alpha_1^{(*)}}{\rightleftarrows}
    V^{(*)}_2
    \stackrel[\beta^{(*)}_2]{\alpha^{(*)}_2}{\rightleftarrows}
    \dots
    \stackrel[\beta^{(*)}_{r-2}]{\alpha^{(*)}_{r-2}}{\rightleftarrows}
    V^{(*)}_{n-1}
    \stackrel[\beta^{(*)}_{r-1}]{\alpha^{(*)}_{r-1}}{\rightleftarrows}
    V^{(*)}_n \leqslant \C^n,  
  \end{equation*} 
  where \( V_j^{(*)}= 0 \) for \( 0 \leqslant j \leqslant k \) and \(
  \alpha^{(*)}_j \) is injective and \( \beta^{(*)}_j \) is surjective
  for \( k<j<n \), and a quiver
  \begin{equation*}
    0 = V^{(0)}_0 \rightleftarrows
    V^{(0)}_1 \rightleftarrows
    V^{(0)}_2 \rightleftarrows \dots
    \rightleftarrows V^{(0)}_{n-1}
    \rightleftarrows V^{(0)}_n = \{ 0 \}
  \end{equation*} 
  in which all maps are \( 0 \). It also follows from the same
theorem that the direct sum of the contracted subquivers is completely
determined (modulo the action of \( \prod_{i: \dim V^j_{i-1} < \dim V^j_i} \GL(V^j_i) \)) by the nilpotent element of \( (\lie{k}_\sim)_\C \) given by the
sum of the complex moment maps \( \alpha_{n-1}^{(*)} \beta_{n-1}^{(*)} \). Furthermore,
given \( \sim \), the adjoint orbit of this nilpotent element in \( (\lie{k}_\sim)_\C \)
corresponds precisely to determining the dimensions of the various vector spaces
\( V_j^{(*)} \) and \( V_j^{(0)} \). 
To see this, observe first that by Remarks \ref{remkappa} and \ref{remdim} the
equivalence relation \( \sim \) determines the dimensions of the generalised eigenspaces of the compositions \( \alpha_i\beta_i \) and determines how the corresponding subquivers are contracted. Also the nilpotent cone for \( (K_\sim)_\C \) in
\( (\lie{k}_\sim)_\C \) is the nilpotent cone for the product 
\( [(K_\sim)_\C,(K_\sim)_\C] \) of special linear groups. Since \( (K_\sim)_\C \) is 
the product of its centre \( Z((K_\sim)_\C) \) and its commutator subgroup
\( [(K_\sim)_\C,(K_\sim)_\C] \), the nilpotent orbits for \( (K_\sim)_\C \) are
the same as the nilpotent orbits for \( [(K_\sim)_\C,(K_\sim)_\C] \), and thus
are given by products of nilpotent orbits in the special linear groups corresponding to the equivalence classes of \( \sim \). These nilpotent orbits in
special linear groups are determined in turn by their Jordan types, which
correspond exactly to the data given by the dimensions of the kernels of their
powers.
The contracted quivers satisfy \( \alpha_i^{(*)} \beta_i^{(*)} =
\beta_{i+1}^{(*)} \alpha_{i+1}^{(*)} \) for all \( i \), and so
\begin{equation*}
(\alpha_{n-1}^{(*)} \beta_{n-1}^{(*)} )^s = 
\alpha_{n-1}^{(*)} \alpha_{n-2}^{(*)}\ldots
\alpha_{n-s}^{(*)} \beta_{n-s}^{(*)} \ldots \beta_{n-2}^{(*)} \beta_{n-1}^{(*)}
.
\end{equation*}
Since the \( \alpha_{i}^{(*)} \) are injective and the \( \beta_{i}^{(*)} \)
are surjective this composition has rank \( \dim V_s^{(*)} \) and
nullity \( \dim V_{n-1}^{(*)} - \dim V_s^{(*)} \). Finally note that the sums
\( \dim V_s^{(*)} + \dim V_s^{(0)} \) are determined by the dimensions of the
vector spaces in the contracted subquiver.
\begin{remark} \label{sumdecomp}
Recall that since our quiver 
\( \q \) lies in \( Q_{(S,\delta)} \) it is the direct sum of a hyperk\"ahler stable quiver of the form
\begin{equation} \label{eq6.8}
        0 \stackrel[\beta_0]{\alpha_0}{\rightleftarrows}
    \C^{m_1}\stackrel[\beta_1]{\alpha_1}{\rightleftarrows}
    \C^{m_2}\stackrel[\beta_2]{\alpha_2}{\rightleftarrows}\dots
    \stackrel[\beta_{n-2}]{\alpha_{n-2}}{\rightleftarrows}
    \C^{m_{n-1}}
    \stackrel[\beta_{n-1}]{\alpha_{n-1}}{\rightleftarrows}
    \C^{m_n} = \C^n
  \end{equation}
 with quivers given
for \( 1
  \leqslant h \leqslant p \) by
  \begin{equation*}
    \C^{d_h}
    \stackrel[\beta_{i_h}^{(h)}]{\alpha_{i_h}^{(h)}}{\rightleftarrows}
    \C^{d_h} \rightleftarrows \dots \rightleftarrows \C^{d_h}
    \stackrel[\beta_{j_h-2}^{(h)}]{\alpha_{j_h-2}^{(h)}}{\rightleftarrows}
    \C^{d_h}    
  \end{equation*}
  in the places \( i_h,i_h + 1, \dots ,j_h -1 \), where the maps \(
  \alpha_k^{(h)} \), \( \beta_k^{(h)} \), for \( i_h \leqslant k < j_h
  - 1 \), are multiplication by complex scalars.
The latter correspond in the description above to the zero summands of the
contracted subquivers, while the former is the direct sum of  the summands
of the form 
  \begin{equation*}
    0=V_0^{(*)}
    \stackrel[\beta^{(*)}_0]{\alpha^{(*)}_0}{\rightleftarrows}
    V^{(*)}_1
    \stackrel[\beta^{(*)}_1]{\alpha_1^{(*)}}{\rightleftarrows}
    V^{(*)}_2
    \stackrel[\beta^{(*)}_2]{\alpha^{(*)}_2}{\rightleftarrows}
    \dots
    \stackrel[\beta^{(*)}_{n-2}]{\alpha^{(*)}_{n-2}}{\rightleftarrows}
    V^{(*)}_{n-1}
    \stackrel[\beta^{(*)}_{n-1}]{\alpha^{(*)}_{n-1}}{\rightleftarrows}
    V^{(*)}_n \leqslant \C^n,  
  \end{equation*} 
  where \( V_j^{(*)}= 0 \) for \( 0 \leqslant j \leqslant k \) and \(
  \alpha^{(*)}_j \) is injective and \( \beta^{(*)}_j \) is surjective
  for \( k<j<n \).
\end{remark}
\begin{remark}
\label{remequiv}
It follows that once the equivalence relation \( \sim \) or its corresponding
subgroup \( K_\sim \) of \( K \) is fixed, the choice of index \( (S,\delta) \) such that 
\( Q_{(S,\delta,\sim)} \) is nonempty corresponds exactly to a nilpotent  adjoint orbit \( \mathcal{O} \subseteq (\lie{k}_\sim)_\C \) 
 for the complexification \( (K_\sim)_\C \) of \( K_\sim \). Thus we can make the following definition.
\end{remark}
\begin{definition} \label{defnsimnil}
Let \( \sim \) be an equivalence relation on \( \{1, \ldots, n\} \) and let 
\( \mathcal{O} \) 
 be a nilpotent adjoint orbit for \( (K_\sim)_\C \). 
 Then we will denote by
\( Q_{[\sim,\mathcal{O}]} \) the subset \( Q_{(S,\delta,\sim)} \) of \( Q \) indexed by
the corresponding \( (S,\delta,\sim) \), and  we will denote by
\( Q_{[\sim,\mathcal{O}]}^\circ  \) its open subset \( Q_{(S,\delta,\sim)}^\circ  \).
\end{definition}

\begin{remark}
We also
remark that the above analysis shows
 that the subset \( Q_{(S,\delta,\sim)} \) of \( Q \) is empty unless the subset \( S \) of
\( \{1,\dots,n\} \times
  \{1,\dots,n\} \) is contained in \( \sim \) (where the equivalence relation \( \sim \) on 
\( \{1,\dots,n\} \) is formally identified with the subset
\begin{equation*}
\{(i,j) \in \{1,\dots,n\} \times
  \{1,\dots,n\}:i \sim j\}
\end{equation*}
of \(   \{1,\dots,n\} \times
  \{1,\dots,n\} \)).
Thus  \( Q_{(S,\delta)} \) is the disjoint union 
\begin{equation*} Q_{(S,\delta)} = \coprod_{\sim} Q_{(S,\delta, \sim)} 
  \end{equation*}
over all the equivalence relations \( \sim \) containing \( S \), and 
  \begin{equation} \label{newstrat}
    Q = \coprod_{S,\delta, \sim} Q_{(S,\delta, \sim)} 
  \end{equation}
  is the disjoint union over all choices of \( S \) and \( \delta
  \) and equivalence relations \( \sim \) containing \( S \). Equivalently 
 \( Q \) is the disjoint union 
\begin{equation*} Q = \coprod_{\sim, \mathcal{O}} Q_{[\sim,\mathcal{O}]} 
  \end{equation*}
over all  equivalence relations \( \sim \) on \( \{1,\ldots,n\} \) and all nilpotent adjoint orbits \( \mathcal{O} \) in 
 \( (\lie{k}_\sim)_\C \).
\end{remark}

\begin{remark} \label{rem5.15}
Notice that the values at a point in \( Q \) of the hyperk\"ahler moment maps
for the actions on \( Q \) of \( K = \SU(n) \) and \( T = (S^1)^{n-1} \) determine the
stratum \( Q_{(S,\delta,\sim)} \) (or equivalently \( Q_{[\sim, \mathcal{O}]} \))
to which the quiver belongs. For the value 
\( (\lambda_1, \dots, \lambda_{n-1}) \) of \( \mu_{(S^1)^{n-1}} \) determines the
equivalence relation \( \sim \) and also the generic choices of complex structures
for which 
  \begin{equation*}
  \sum_{k=i}^{j-1} \lambda_k = 0
  \ \text{in}\ \R^3 \iff \sum_{k=i}^{j-1} \lambda_k^\C = 0
  \ \text{in}\ \C.
\end{equation*}
Moreover for such choices of complex structures the quiver decomposes as a direct sum of subquivers determined by the generalised eigenspaces of the 
composition \( \alpha_{n-1}\beta_{n-1} \) (which is given by the complex
moment map for the action of \( K \)), and
it follows 
 that the Jordan type of \( \alpha_{n-1}\beta_{n-1} \)
determines the nilpotent orbit \( \mathcal{O} \) in \( (\lie{k}_\sim)_\C \),
or equivalently as above
 the data \( S \) and \( \delta \).
\end{remark}

\section{Using Jordan canonical form}
In this section we will use Jordan canonical form to study hyperk\"ahler quivers
as at (\ref{eq6.8}) in order to find standard forms in the next section for quivers in a corresponding stratum \( Q_{[\sim,\mathcal{O}]} \).
We have the following description of such quivers from \cite{DKS} Proposition 7.2.

\begin{proposition}
  \label{betasurj} Let \( 0=m_0 \leqslant m_1 \leqslant 
 \cdots \leqslant m_n = n \) and let \( V_i = \C^{m_i} \) for \( 0 \leqslant i \leqslant n \).
  Consider quivers 
of the form 
  \begin{equation*}
    0=V_0
    \stackrel[\beta_0]{\alpha_0}{\rightleftarrows}
    V_1
    \stackrel[\beta_1]{\alpha_1}{\rightleftarrows}
    V_2
    \stackrel[\beta_2]{\alpha_2}{\rightleftarrows}
    \dots
    \stackrel[\beta_{n-2}]{\alpha_{n-2}}{\rightleftarrows}
    V_{n-1}
    \stackrel[\beta_{n-1}]{\alpha_{n-1}}{\rightleftarrows}
    V_n = \C^n,  
  \end{equation*} 
  where 
each  \( \beta_j \) is surjective 
 and  the complex
  moment map equations for \( H=\prod_{i=1}^{n-1} \SU(m_i) \) are satisfied.
The set of such quivers 
  modulo the action of \( H_\C=\prod_{i=1}^{n-1} \SL(m_i, \C) \) may
  be identified with
  \begin{equation*}
    K_\C  \times_{[P,P]} [\p, \p]^\circ
  \end{equation*}
  where \( P \) is the parabolic subgroup of \( K_\C = \SL(n, \C) \) associated to the flag \(
  (m_{1}, \dots, m_n=n) \), and \( [\p, \p]^\circ \) is the annihilator
  of the Lie algebra of the commutator subgroup of \( P \). The same is true if we replace the assumption that each  \( \beta_j \) is surjective with the
assumption that each  \( \alpha_j \) is injective.
  When 
 \( m_i =i \) for all \( i \)
  we have the space
  \begin{equation*}
    K_\C  \times_N \bmf
  \end{equation*}
  where \( N \) is a maximal unipotent subgroup of \( K_\C = \SL(n,\C) \) and
  \( \bmf = \n^\circ \) is a Borel subalgebra.
\end{proposition}
We can modify this result as follows to describe the subset of quivers for which 
each \( \alpha_j \) is injective and each \( \beta_j \) is surjective.

\begin{proposition} \label{corBQ}
Let \( 0=m_0 \leqslant m_1 \leqslant 
 \cdots \leqslant m_n = n \) and let \( V_i = \C^{m_i} \) for \( 0 \leqslant i \leqslant n \).
  Consider quivers 
of the form 
  \begin{equation*}
    0=V_0
    \stackrel[\beta_0]{\alpha_0}{\rightleftarrows}
    V_1
    \stackrel[\beta_1]{\alpha_1}{\rightleftarrows}
    V_2
    \stackrel[\beta_2]{\alpha_2}{\rightleftarrows}
    \dots
    \stackrel[\beta_{n-2}]{\alpha_{n-2}}{\rightleftarrows}
    V_{n-1}
    \stackrel[\beta_{n-1}]{\alpha_{n-1}}{\rightleftarrows}
    V_n = \C^n,  
  \end{equation*} 
  where 
each \( \alpha_j \) is injective and each \( \beta_j \) is surjective 
 and  the complex
  moment map equations for \( H=\prod_{i=1}^{n-1} \SU(m_i) \) are satisfied.
The set of such quivers 
  modulo the action of \( H_\C=\prod_{i=1}^{n-1} \SL(m_i, \C) \) may
  be identified with
  \begin{equation*}
    K_\C  \times_{[P,P]} [\p, \p]^\circ_* 
  \end{equation*}
  where  \( [\p, \p]^\circ_* \) is the open subset of  \( [\p, \p]^\circ \)
consisting of those  \( X \in [\p, \p]^\circ \) such that 
\begin{equation*}
     X_i - \frac{\mathrm{tr}X_{ii}}{k_i} \left( 
\begin{array}{c} 0_{k_i \times m_i} \\ I_{m_i \times m_i}
\end{array} \right)
  \end{equation*}
has maximal rank \( m_i \) for each \( i \) where \( k_i = m_{i+1} - m_i \).
Here \( X_i \) is the bottom right \( m_{i+1} \times m_i \) block in \( X \)
and \( X_{ii} \) is its \( i \)th diagonal block (of size \( k_i \times k_i \)).
\end{proposition}

\begin{proof} We first recall from \cite{DKS} the proof of 
Proposition \ref{betasurj} above. Given a quiver of the form 
  \begin{equation*}
    0=V_0
    \stackrel[\beta_0]{\alpha_0}{\rightleftarrows}
    V_1
    \stackrel[\beta_1]{\alpha_1}{\rightleftarrows}
    V_2
    \stackrel[\beta_2]{\alpha_2}{\rightleftarrows}
    \dots
    \stackrel[\beta_{n-2}]{\alpha_{n-2}}{\rightleftarrows}
    V_{n-1}
    \stackrel[\beta_{n-1}]{\alpha_{n-1}}{\rightleftarrows}
    V_n = \C^n,  
  \end{equation*} 
  where each \( \beta_j \) is surjective and the complex moment map
  equations for \( H=\prod_{i=k+1}^{n-1} \SU(m_i) \) are satisfied, it
  follows from an easy inductive argument that the vector spaces \(
  V_i \) with \( k<i\leqslant n \) have bases so that
  \begin{equation*}
    \beta_i = \left( 0_{m_i \times k_i} \mid I_{m_i \times m_i} \right),
  \end{equation*}
  where \( m_i = \dim V_i \) and \( k_i = m_{i+1} - m_i \) is the
  dimension of the kernel of \( \beta_i \). We have thus used the action of \( H_\C \times K_\C
  = \prod_{i=1}^n \SL(m_i, \C) \) to put the maps \( \beta_i \) in standard form.
 
This standard form is preserved by transformations satisfying
  \begin{equation*}
    g_{i+1} =
    \begin{pmatrix}
      * & * \\
      0 & g_i
    \end{pmatrix}
    ,
  \end{equation*}
  where the top left block is \( k_i \times k_i \), the bottom
  right block is \( m_i \times m_i \) and \( g_{k+1} \) is an arbitrary
  element of \( \SL(m_{k+1}, \C) \). The freedom
  in \( \SL(n, \C) \) is therefore the commutator of the parabolic group \( P \)
  associated to the flag of dimensions \( (m_{k+1}, m_{k+2}, \dots, m_{n}=n)
  \) in \( \C^n \).

With respect to bases chosen as
above, the matrix of \( \alpha_i \beta_i \) for \( k<i<n \) is
\begin{equation}
  \label{abform}
  \begin{pmatrix}
    0_{k_i \times k_i} & D_{k_i \times m_i} \\
    0_{m_i \times k_i} & -\lambda_i^\C I_{m_i} + \alpha_{i-1}
    \beta_{i-1}
  \end{pmatrix}
\end{equation}
for some \( k_i \times m_i \) matrix \( D \).

We can use the pairing \( (A,B)
\mapsto \tr(AB) \) to identify \( \kf \) and \( \kf_\C \) with their
duals. 
It now follows inductively that  \( X = \alpha_{n-1} \beta_{n-1} \) lies in the
annihilator of the Lie algebra of the commutator \( [P,P] \) of the
parabolic determined by the integers \( k_j \), and the diagonal entries of \( X \) are \( 0
\) (\(k_{n-1} \) times), \( -\lambda_{n-1}^\C \) (\(k_{n-2} \)
times), \( \dots, -(\lambda_{n-1}^\C + \dots + \lambda_{i+1}^\C) \)
(\(k_i \) times),  \( \dots \)

Moreover one can show that 
any such \( X \) occurs for a solution of the complex moment map equations,
and that the trace-free part of \( X \) determines all
the \( \alpha_i \) and hence the entire quiver modulo the action
of \( H_\C \).
Note that 
 if \( X \in  [\p, \p]^\circ \) then the corresponding quiver 
  \begin{equation*}
    0=V_0
    \stackrel[\beta_0]{\alpha_0}{\rightleftarrows}
    V_1
    \stackrel[\beta_1]{\alpha_1}{\rightleftarrows}
    V_2
    \stackrel[\beta_2]{\alpha_2}{\rightleftarrows}
    \dots
    \stackrel[\beta_{n-2}]{\alpha_{n-2}}{\rightleftarrows}
    V_{n-1}
    \stackrel[\beta_{n-1}]{\alpha_{n-1}}{\rightleftarrows}
    V_n = \C^n,  
  \end{equation*} 
with each \( \beta_i \) in standardised form
  \begin{equation*}
    \beta_i = \left( 0_{n_i \times k_i} \mid I_{m_i \times m_i} \right),
  \end{equation*}
has 
  \begin{equation*}
    \alpha_i = X_i - \frac{\mathrm{tr}X_{ii}}{k_i} \left( 
\begin{array}{c} 0_{k_i \times m_i} \\ I_{m_i \times m_i}
\end{array} \right)
  \end{equation*}
where \( X_i \) is the bottom right \( m_{i+1} \times m_i \) block in \( X \)
and \( X_{ii} \) is its \( i \)th diagonal block (of size \( k_i \times k_i \)).

If we sweep out the space \( [\p,\p]_*^\circ \) of quivers of this form with
each \( \alpha_j \) injective by the
action of the torus \( T_\C = (\C^*)^{n-1} \) then we obtain the space \( T_\C [\p,\p]_*^\circ \)
of quivers for which \( X = \alpha_{n-1}\beta_{n-1} \in [\p,\p]_*^\circ \). Here
each \( \beta_k \) can be put in the form
\begin{equation*}
  \beta_i = \left( 0_{m_i \times k_i} \mid a_i I_{m_i \times m_i} \right),
\end{equation*}
for some nonzero scalar \( a_i \) by using the action of \( H_\C=\prod_{i=k+1}^{n-1} \SL(m_i, \C) \) but \emph{not} the action of \( K_\C = \SL(n,\C) \). Then the
set of all quivers 
of the form 
  \begin{equation*}
    0=V_0
    \stackrel[\beta_0]{\alpha_0}{\rightleftarrows}
    V_1
    \stackrel[\beta_1]{\alpha_1}{\rightleftarrows}
    V_2
    \stackrel[\beta_2]{\alpha_2}{\rightleftarrows}
    \dots
    \stackrel[\beta_{n-2}]{\alpha_{n-2}}{\rightleftarrows}
    V_{n-1}
    \stackrel[\beta_{n-1}]{\alpha_{n-1}}{\rightleftarrows}
    V_n = \C^n,  
  \end{equation*} 
  where 
each  \( \beta_j \) is surjective and each \( \alpha_j \) is injective 
 and  the complex
  moment map equations for \( H=\prod_{i=k+1}^{n-1} \SU(m_i) \) are satisfied,
  modulo the action of \( H_\C=\prod_{i=k+1}^{n-1} \SL(m_i, \C) \), is
identified with
  \begin{equation*}
    K_\C  \times_{[P,P]} [\p, \p]_*^\circ \cong K_\C \times_P (T_\C [\p,\p]_*^\circ)
  \end{equation*}
since \( P = T_\C [P,P] \).
This completes the proof.
\end{proof}
\begin{remark} \label{remJCF}
Let  
  \begin{equation*}
    0=V_0
    \stackrel[\beta_0]{\alpha_0}{\rightleftarrows}
    V_1
    \stackrel[\beta_1]{\alpha_1}{\rightleftarrows}
    V_2
    \stackrel[\beta_2]{\alpha_2}{\rightleftarrows}
    \dots
    \stackrel[\beta_{n-2}]{\alpha_{n-2}}{\rightleftarrows}
    V_{n-1}
    \stackrel[\beta_{n-1}]{\alpha_{n-1}}{\rightleftarrows}
    V_n = \C^n,  
  \end{equation*} 
be a quiver such that 
each \( \alpha_j \) is injective and each \( \beta_j \) is surjective 
 and  the complex
  moment map equations for \( H=\prod_{i=k+1}^{n-1} \SU(m_i) \) are satisfied. 
Then \( \ker 
  \alpha_j  \beta_j  = \ker   \beta_j \)
and   \( \mathrm{im} \,
  \alpha_j  \beta_j  = \mathrm{im} \,  \alpha_j \) for each \( j \),
and we can inductively choose coordinates on \( V_1, \ldots V_n \), starting
with \( V_n = \C^n \), so that each \( \alpha_{j}\beta_{j} \) (and hence each \( \beta_j \alpha_j \))
is in Jordan canonical form while each \( \beta_j \) is 
the direct sum over all the Jordan blocks of matrices in the standardised form
  \begin{equation*}
 \left( \begin{array}{cccccc}
0 & 1 & 0 & 0 & \cdots & 0\\
0 & 0 & 1 & 0 & \cdots & 0\\
 & & \cdots & & & \\
0 & \cdots & 0 &  0 & 1 & 0 \\
0 & \cdots & 0 & 0 & 0 & 1
  \end{array} \right).
  \end{equation*}
More precisely, we first choose a basis \( \{e_1,\ldots,e_n\} \) for \( V_n = \C^n \) so 
that \( \alpha_{n-1} \beta_{n-1} \) is in Jordan canonical form. Since
\( \beta_{n-1} :V_n \to V_{n-1} \) is surjective and \( \ker \beta_{n-1} =
\ker \alpha_{n-1} \beta_{n-1} \), it follows that
\begin{equation*}
\{ \beta_{n-1}(e_i): e_i \not\in \ker \beta_{n-1} \}
\end{equation*}
is a basis for \( V_{n-1} \) such that \( \beta_{n-1} \) is in the required form. It
is now easy to check that \( \beta_{n-1} \alpha_{n-1} \) is in Jordan canonical form,
and it follows immediately from the complex moment map equations that the same is true of \( \alpha_{n-2} \beta_{n-2} \), so that we can repeat the argument with
the basis 
\begin{equation*}
\{ \beta_{n-2}\beta_{n-1}(e_i): e_i \not\in \ker \beta_{n-2} \beta_{n-1} \}
\end{equation*}
 for \( V_{n-2} \).
Alternatively we can choose coordinates so that  each \( \alpha_{j}\beta_{j} \)
and \( \beta_j\alpha_j \)
is in Jordan canonical form while each \( \alpha_j \) is 
the direct sum over all the Jordan blocks of matrices in the standardised form
  \begin{equation*}
 \left( \begin{array}{ccccc}
1 & 0 & 0 & \cdots & 0\\
 0 & 1 & 0 & \cdots & 0\\
& & \cdots & &\\
0 & \cdots & 0  & 1 & 0 \\
0 & \cdots & 0  & 0 & 1 \\
0 & \cdots & 0  & 0 & 0
  \end{array} \right).
  \end{equation*}
Note that if we write elements \( \zeta \) of \( \lie{k}_\C \) in the form \( \zeta = \zeta_{\bf s} + \zeta_{\bf n} \) where \( \zeta_{\bf s} \) is semisimple and \( \zeta_{\bf n} \) is nilpotent and \( [\zeta_{\bf s},\zeta_{\bf n}]=0 \) (cf.\! \cite{Collingwood-McGovern} \S1.1), then 
\( (\alpha_{n-1}\beta_{n-1})_{\bf n} \in [\p, \p]_* \) and the
Jordan type of \( \alpha_{n-1}\beta_{n-1} \) determines, and is determined
by, the flag \(
  (m_{1}, \dots, m_n=n) \) (or equivalently the parabolic subgroup \( P \) of \( K_\C = \SL(n, \C) \) ) together with the nilpotent orbit containing \( (\alpha_{n-1}\beta_{n-1})_{\bf n} \).
Note also that the fact that \( \alpha_{n-1}\beta_{n-1} \) lies in the open subset
\( [\p, \p]_* \) of \( [\p,\p] \) tells us that  \( P \) is the Jacobson-Morozov parabolic
of \( (\alpha_{n-1}\beta_{n-1})_{\bf n} \) (cf.\! \cite{Collingwood-McGovern} \S3.8).
The choice of coordinates needed to put the quiver into the standard form above
amounts to using the action of the group 
\(  K_\C \times \prod_{j=1}^{n-1} \GL(m_j,\C) \) to standardise the quiver.
Equivalently, once we have quotiented by the action of  \( \prod_{j=1}^{n-1} \SL(m_j,\C) \), it amounts to using
 the action of \( K_\C \times (\C^*)^{n-1} = K_\C \times T_\C \) to 
standardise the point of \( Q \) represented by the quiver.
\end{remark}

\section{Standard forms for quivers} 
Let \( \sim \) be an equivalence relation on \( \{1, \ldots, n\} \), and let 
\( \mathcal{O}  \) be a nilpotent adjoint orbit 
 in \( (\lie{k}_\sim)_\C \). Recall that 
\( Q_{[\sim,\mathcal{O}]} \) is the subset \( Q_{(S,\delta,\sim)} \) of \( Q \) indexed by
the corresponding \( (S,\delta,\sim) \) as in Definition \ref{defnsimnil}. Similarly
we will write \( Q_{[\sim,\mathcal{O}]}^\circ \) for the open  subset  \( Q_{(S,\delta,\sim)}^\circ \) of \( Q_{(S,\delta,\sim)} \).
Putting the results of the last two sections together we find that given any quiver \( \q \) in \( Q_{[\sim,\mathcal{O}]}^\circ \) we can choose coordinates on \( \C, \C^2, \ldots, \C^n \) 
to put the quiver into standard form as
follows. Firstly, 
 as in Remark~\ref{lem5.14} 
we can decompose \( \q \) into  a direct sum of subquivers
  \begin{equation*} \cdots 
    V_i^j
    \stackrel[\beta_{i,j}]{\alpha_{i,j}}{\rightleftarrows}
    V_{i+1}^j \cdots 
  \end{equation*}
determined by the generalised eigenspaces 
of the compositions \( \alpha_i\beta_i \). 
Since \( \q \) lies in \( Q_{[\sim,\mathcal{O}]}^\circ  \) each such
subquiver is the  direct sum of a quiver \( \q^{[j]} \) of the form
\begin{equation}  \label{eeq6.8}
        0 \stackrel[\beta^{[j]}_0]{\alpha^{[j]}_0}{\rightleftarrows}
    \C^{m_1}\stackrel[\beta_1^{[j]}]{\alpha_1^{[j]}}{\rightleftarrows}
    \C^{m_2}\stackrel[\beta_2^{[j]}]{\alpha_2^{[j]}}{\rightleftarrows}\dots
    \stackrel[\beta_{n-2}^{[j]}]{\alpha_{n-2}^{[j]}}{\rightleftarrows}
    \C^{m_{n-1}}
    \stackrel[\beta_{n-1}^{[j]}]{\alpha_{n-1}^{[j]}}{\rightleftarrows}
    \C^{m_n} 
  \end{equation}
 where the maps 
\( \alpha_k^{[j]} \) for \( 1\leqslant k \leqslant n-1 \) are injective and the
maps \( \beta_k^{[j]} \)  for \( 1\leqslant k \leqslant n-1 \) are surjective, with quivers 
given 
for \( 1
  \leqslant h \leqslant p \) by
  \begin{equation*}
    \C^{d_h}
    \stackrel[\beta_{i_h}^{(h)}]{\alpha_{i_h}^{(h)}}{\rightleftarrows}
    \C^{d_h} \rightleftarrows \dots \rightleftarrows \C^{d_h}
    \stackrel[\beta_{j_h-2}^{(h)}]{\alpha_{j_h-2}^{(h)}}{\rightleftarrows}
    \C^{d_h}    
  \end{equation*}
  in the places \( i_h,i_h + 1, \dots ,j_h -1 \), where the
maps
\(
  \alpha_k^{(h)} \), \( \beta_k^{(h)} \), for \( i_h \leqslant k < j_h
  - 1 \), are multiplication by complex scalars 
such that \( \gamma_k^{(h)} = \alpha_k^{(h)} + j \beta_k^{(h)} \in \mathbb{H}\setminus \{0\} \) (cf. Remark \ref{sumdecomp}). Moreover the combinatorial data
here and the Jordan type of \( \alpha_{n-1}\beta_{n-1} \) for each summand (\ref{eeq6.8}) is determined by the pair \( (\sim,\mathcal{O}) \) (see Remarks
\ref{remequiv} and \ref{remJCF}). 
Now if we allow any complex linear changes of coordinates in
\( K_\C \times H_\C = \prod_{k=1}^n \SL(k,\C) \)  then 
using Remark \ref{remJCF} we can
put \( \alpha_{n-1}\beta_{n-1} \) into Jordan canonical form and then
 decompose the quiver (\ref{eeq6.8})
into a direct sum of quivers determined by the Jordan blocks of \( \alpha^{[j]}_{n-1}\beta^{[j]}_{n-1} \). Now
 \( \alpha^{[j]}_k \) is a direct sum over the set \( B_j \) of 
Jordan blocks for \( \alpha^{[j]}_{n-1}\beta^{[j]}_{n-1} \)
of matrices of the form
\begin{equation} \label{formstar}
 \left( \begin{array}{ccccc}
\nu_1^{bjk} & 0 & 0 & \cdots & 0\\
0 & \nu_2^{bjk} & 0 & \cdots & 0\\
 & & \cdots & & \\
0 & \cdots & 0 & 0 & \nu^{bjk}_{\ell_b - n+k}\\
0 & \cdots & 0 & 0 & 0 \end{array} \right) \end{equation}
for some \( \nu^{bjk}_i \in\C^* \) where \( \ell_b \) is the size
of the Jordan block \( b \in B_j \), and \( \beta_k^{[j]} \) is a corresponding direct sum over \( b \in B_j \) of matrices of the form
\begin{equation*}
\left( \begin{array}{cccccc}
\mu_1^{bjk} & \xi_1^{bjk} & 0 & 0 & \cdots & 0\\
0 & \mu_2^{bjk} & \xi_2^{bjk} & 0 & \cdots & 0\\
 & & \cdots & & & \\
0 & \cdots & 0 &  \mu^{bjk}_{\ell_b - n + k -1} & \xi^{bjk}_{\ell_b - n +k -1} & 0 \\
0 & \cdots & 0 & 0 & \mu^{bjk}_{\ell_b - n+k} & \xi^{bjk}_{\ell_b - n +k} 
  \end{array} \right)
\end{equation*}
for some \( \mu_i^{bjk}, \xi_i^{bjk} \in\C^* \) 
 satisfying the complex
moment map equations (\ref{eq:mmcomplex}). The resulting direct
sum over all the Jordan blocks \( \bigcup_{j}B_j \) 
for \( \alpha_{n-1}\beta_{n-1} \) has closed \( (H_S)_\C \)-orbit. 
Let \( Q_{[\sim,\mathcal{O}]}^{\circ,JCF} \) be the subset of \( Q_{[\sim,\mathcal{O}]}^{\circ} \) representing quivers of this form, 
where  \( \alpha_{n-1}\beta_{n-1} \) is in Jordan canonical form and
 the summands of the quiver corresponding to generalised eigenspaces of the compositions
\( \alpha_i \beta_i \) (and
thus by Remark \ref{remkappa} to equivalence classes for \( \sim \)) are ordered according to the
usual ordering on the minimal elements of the equivalence classes, and
the Jordan blocks for each equivalence class are ordered by size. Then
 we have
\begin{equation*}
Q_{[\sim,\mathcal{O}]}^\circ = K_\C Q_{[\sim,\mathcal{O}]}^{\circ,JCF}.
\end{equation*}
Note that if we allow any complex linear changes of coordinates in
\( K_\C \times \tilde{H}_\C = \prod_{k=1}^n \GL(k,\C) \) (or equivalently
allow the action of \( K_\C \times T_\C \) on \( Q_{[\sim,\mathcal{O}]} \)), then 
as in  Remark \ref{remJCF}
we can put our quiver into a more restricted form which is completely
determined by \( \alpha_{n-1}\beta_{n-1} \) and \( (\lambda_1^\C, \ldots, \lambda_{n-1}^\C) \),
and thus by the values of the
complex moment maps for the actions of \( K \) and \( T \) on \( Q \).
Hence 
the fibres of the complex moment map
\begin{equation*}
Q_{[\sim,\mathcal{O}]}^{\circ} \to \lie{k}_\C \oplus \lie{t}_\C
\end{equation*}
for the action of \( K \times T \) are contained in single \( K_\C \times T_\C \)-orbits.
 
\begin{remark}
 \label{remchoice}
Let us consider the stabiliser in \( K_\C \times T_\C \) of a quiver \( \q \in 
Q_{[\sim,\mathcal{O}]}^{\circ} \). We may assume that \( \q \) is in the standard form
described above, so that  \( \q \in 
Q_{[\sim,\mathcal{O}]}^{\circ,JCF} \). We also want to consider how much of the
\( K_\C \times T_\C \)-orbit of \( \q \) lies in \( Q_{[\sim,\mathcal{O}]}^{\circ,JCF} \).

We know from Remark~\ref{remkappa} that the decomposition of \( \q \) into  a direct sum of subquivers
 given by the generalised eigenspaces 
of the compositions \( \alpha_i\beta_i \) is determined by 
the equivalence relation \( \sim \), and since this decomposition is
canonical it follows that the stabiliser of \( \q \) in \( K_\C \times T_\C \)
is a subgroup of \( (K_\sim)_\C \times T_\C \), and indeed that elements of
\( K_\C \times T_\C \) which preserve the standard form all lie in
\( (K_\sim)_\C \times T_\C \). Now applying the proof of Proposition \ref{corBQ}
and Remark \ref{remJCF} to 
the summands in this decomposition, it follows that elements of \( K_\C \times T_\C \) which preserve the standard form are
contained in \( P \times T_\C \) where \( P \) is the parabolic subgroup of \( (K_\sim)_\C \)
which is the Jacobson--Morozov parabolic of the element of the nilpotent 
orbit \( \mathcal{O} \)
for  \( (K_\sim)_\C \) given by the nilpotent component of 
\begin{equation*}
    \alpha_{n-1} \beta_{n-1} - \frac1n \tr(\alpha_{n-1} \beta_{n-1}) I_n
\in (\lie{k}_\sim)_\C.
  \end{equation*}
Indeed,  the elements which preserve the standard form must lie in \( R_{[\sim,\mathcal{O}]} \times T_\C \) where
\( R_{[\sim,\mathcal{O}]} \) is the centraliser in \( P \) of this nilpotent element of
\( [\lie{p},\lie{p}] \). In particular the stabiliser of \( \q \) in \( K_\C \times T_\C \) is
contained in \( R_{[\sim,\mathcal{O}]} \times T_\C \).
Conversely, it is easy to see that both the centre
\( Z((K_\sim)_\C) \) of \( (K_\sim)_\C \), embedded diagonally in 
\( T_\C \times T_\C \), and the intersection
\( [P,P] \cap R_{[\sim,\mathcal{O}]} \), embedded in \( K_\C \times \{1\} \), stabilise the quiver \( \q \in Q_{[\sim,\mathcal{O}]}^{\circ,JCF} \).  This quiver is also stabilised by 
the complexification of the subgroup \( T_S \) defined at Definition 4.2.
Moreover since \( Q_{[\sim,\mathcal{O}]}^{\circ,JCF}  \) is 
\( R_{[\sim,\mathcal{O}]} \times T_\C \)-invariant it follows that
an element of \( K_\C \times T_\C \) preserves the standard form if and only if it
lies in  \( R_{[\sim,\mathcal{O}]} \times T_\C \), and so
\begin{equation} \label{identity} Q_{[\sim,\mathcal{O}]}^\circ
\cong  (K_\C \times T_\C) \times_{(R_{[\sim,\mathcal{O}]} \times T_\C) } Q_{[\sim,\mathcal{O}]}^{\circ,JCF}
 \cong  K_\C \times_{R_{[\sim,\mathcal{O}]}  } Q_{[\sim,\mathcal{O}]}^{\circ,JCF}. \end{equation}
Furthermore, to determine the stabiliser of \( \q \) in \( K_\C \times T_\C \) we
should first consider its intersection with \( (T_{[\sim,\mathcal{O}]})_\C \times T_\C \), where \( T_{[\sim,\mathcal{O}]} \) is
the intersection of \( T \) with \( R_{[\sim,\mathcal{O}]} \).
The intersection of the stabiliser of \( \q \) with 
\( (T_{[\sim,\mathcal{O}]})_\C \times T_\C \) contains 
the product of the subgroups 
\( Z((K_\sim)_\C) \), \( Z_P \cap  R_{[\sim,\mathcal{O}]} \)
and \( (T_S)_\C \) of \( (T_{[\sim,\mathcal{O}]})_\C \times T_\C \).
\end{remark}

\begin{lemma} \label{leminj}
The non-empty fibres of the restriction
\begin{equation*}
Q_{[\sim,\mathcal{O}]}^{\circ,JCF} \to \lie{k}_\C \oplus \lie{t}_\C
\end{equation*}
to \( Q_{[\sim,\mathcal{O}]}^{\circ,JCF}  \) of the complex moment map
for the action of \( K \times T \) on \( Q \) are single \( (T_{[\sim,\mathcal{O}]})_\C \times T_\C \)-orbits,  where \( T_{[\sim,\mathcal{O}]} \) is
the intersection of \( T \) with \( R_{[\sim,\mathcal{O}]} \).
\end{lemma}
\begin{proof}
We observed 
just before Remark \ref{remchoice} that the fibres of the complex moment map
\begin{equation*}
Q_{[\sim,\mathcal{O}]}^{\circ} \to \lie{k}_\C \oplus \lie{t}_\C
\end{equation*}
for the action of \( K \times T \) are contained in single \( K_\C \times T_\C \)-orbits.
 By (\ref{identity})
above, each \( K_\C \times T_\C \)-orbit which 
 meets \( Q_{[\sim,\mathcal{O}]}^{\circ,JCF} \) meets it in a single \( R_{[\sim,\mathcal{O}]} \times T_\C \)-orbit where
\begin{equation*}
R_{[\sim,\mathcal{O}]} = (T_{[\sim,\mathcal{O}]})_\C \,\,([P,P] \cap R_{[\sim,\mathcal{O}]})
\end{equation*}
and 
\( [P,P] \cap R_{[\sim,\mathcal{O}]} \) stabilises the quiver \( \q \in Q_{[\sim,\mathcal{O}]}^{\circ,JCF} \). Thus each fibre  of the restriction
to \( Q_{[\sim,\mathcal{O}]}^{\circ,JCF}  \) of the complex moment map
for the action of \( K \times T \) is contained in a single  \( (T_{[\sim,\mathcal{O}]})_\C \times T_\C \)-orbit. Since this complex moment map is
\( K_\C \times T_\C \)-equivariant and  \( (T_{[\sim,\mathcal{O}]})_\C \times T_\C \) fixes
the image in \(  \lie{k}_\C \oplus \lie{t}_\C \) of any element of \( Q_{[\sim,\mathcal{O}]}^{\circ,JCF} \), the result follows.
\end{proof}
\begin{remark}
\label{rem7.5a} Recall that the subgroup \( [P,P] \cap (T_{[\sim,\mathcal{O}]})_\C  \) acts trivially on \( Q_{[\sim,\mathcal{O}]}^{\circ,JCF} \). In fact 
\(  (T_{[\sim,\mathcal{O}]})_\C/[P,P] \cap  (T_{[\sim,\mathcal{O}]})_\C \) acts freely on  \( Q_{[\sim,\mathcal{O}]}^{\circ,JCF} \). To see this, consider 
a quiver which is a direct sum over the set \( \bigcup_{j}B_j \) of Jordan blocks for
\( \alpha_{n-1}\beta_{n-1} \) of quivers of the form described in 
(\ref{formstar}), together with quivers given
for \( 1
  \leqslant h \leqslant p \) by
  \begin{equation*}
    \C^{d_h}
    \stackrel[\beta_{i_h}^{(h)}]{\alpha_{i_h}^{(h)}}{\rightleftarrows}
    \C^{d_h} \rightleftarrows \dots \rightleftarrows \C^{d_h}
    \stackrel[\beta_{j_h-2}^{(h)}]{\alpha_{j_h-2}^{(h)}}{\rightleftarrows}
    \C^{d_h}    
  \end{equation*}
  in the places \( i_h,i_h + 1, \dots ,j_h -1 \), as in Remark \ref{sumdecomp}. 
The centraliser in \( T_\C \) of \( \alpha_{n-1}\beta_{n-1} \) consists of matrices in \( T_\C \) which are themselves direct sums over all the Jordan blocks of diagonal 
matrices with diagonal entries \( (\tau^{bjk}_1, \ldots , \tau^{bjk}_1,\tau^{bjk}_2) \) for some \( \tau^{bjk}_1,\tau^{bjk}_2 \in \C^* \). If
such a matrix sends the quiver to an element of its \( H \)-orbit, then \( \tau^{bjk}_1 = \tau^{bjk}_2 \) for all \( b,j,k \), and so the matrix lies in
\( [P,P] \cap (T_{[\sim,\mathcal{O}]})_\C  \).
\end{remark}
\begin{remark} \label{remsigma2}
The image  in
\(  \lie{k}_\C \) of \( \q \in Q_{[\sim,\mathcal{O}]}^{\circ} \) under the complex moment map for \( K \)
is the sum of an element \( \zeta \) of \( \lie{t}_\C \) with centraliser \( K_\sim \) in \( K \) and an element \( \xi \) of the nilpotent orbit \( \mathcal{O} \) in \( (\lie{k}_\sim)_\C \), and its image in \(  \lie{t}_\C \) under
the complex moment map for \( T \)
is equal to \( \zeta \) by Remark \ref{remkappa}. 

Hence the image of 
\( Q_{[\sim,\mathcal{O}]}^{\circ,JCF} \)  under the complex
moment map for \( K \times T \) is
\begin{equation*}
  \Delta (\lie{t}_\C)_\sim \oplus \xi_0 = 
\{ (\zeta + \xi_0, \zeta) \in \lie{k}_\C \oplus \lie{t}_\C : \zeta \in 
(\lie{t}_\C)_\sim \},
\end{equation*}
where \( (\lie{t}_\C)_\sim \) consists of the elements of \(
\lie{t}_\C \) with centraliser \( K_\sim \)
 in \( K \).
Also \( \xi_0 \in \mathcal{O} \) is the element of the nilpotent orbit
\( \mathcal{O} \) in Jordan canonical form with the  Jordan blocks 
within each generalised eigenspace for \( \alpha_{n-1}\beta_{n-1} \) ordered by size
and the generalised eigenspaces themselves ordered
using the equivalence relation \( \sim \). 

The image of \( Q_{[\sim,\mathcal{O}]}^{\circ} \)  under this
moment map is the \( K_{\C} \) sweep of 
\( \Delta (\lie{t}_\C)_\sim \oplus \xi_0 \), where \( K_{\C} \) acts trivially
on the second component.

\end{remark}

\begin{remark} \label{remqt}
Notice that if a quiver is of 
 the form
\begin{equation*}
\alpha_k^{bj} =
 \left( \begin{array}{ccccc}
\nu_1^{bjk} & 0 & 0 & \cdots & 0\\
0 & \nu_2^{bjk} & 0 & \cdots & 0\\
 & & \cdots & & \\
0 & \cdots & 0 & 0 & \nu^{bjk}_{\ell_b - n+k}\\
0 & \cdots & 0 & 0 & 0 \end{array} \right) \end{equation*}
and 
\begin{equation*}
\beta_k^{bj}=\left( \begin{array}{cccccc}
\mu^{bjk}_1 & \xi^{bjk}_1 & 0 & 0 & \cdots & 0\\
0 & \mu^{bjk}_2 & \xi^{bjk}_2 & 0 & \cdots & 0\\
 & & \cdots & & & \\
0 & \cdots & 0 &  \mu^{bjk}_{\ell_b - n + k -1} & \xi^{bjk}_{\ell_b - n + k -1} & 0 \\
0 & \cdots & 0 & 0 & \mu^{bjk}_{\ell_b - n + k} & \xi^{bjk}_{\ell-n+k} 
  \end{array} \right)
\end{equation*}
for some \( \nu_i^{bjk}, \mu_i^{bjk}, \xi_i^{bjk} \in\C^* \) as at
(\ref{formstar}) above,
then \( \alpha^{bj}_k\beta^{bj}_k \) and \( \beta^{bj}_k \alpha^{bj}_k \) are upper triangular matrices with
diagonal entries
\begin{equation*}
\mu_1^{bjk}\nu_1^{bjk}, \ldots, \mu^{bjk}_{\ell_b - n + k}\nu_{\ell_b - n + k}^{bjk}, 0
\end{equation*}
and 
\begin{equation*}
\mu_1^{bjk}\nu_1^{bjk}, \ldots, \mu^{bjk}_{\ell_b - n + k}\nu_{\ell_b - n + k}^{bjk}
\end{equation*}
respectively. It follows that the complex moment map equations are also
satisfied by the quiver given by replacing \( \alpha_k^{bj} \) and \( \beta_k^{bj} \) with
\begin{equation} \label{diagquiv}
 \alpha_k^{bj,T} = \left( \begin{array}{ccccc}
\nu_1^{bjk} & 0 & 0 & \cdots & 0\\
0 & \nu_2^{bjk} & 0 & \cdots & 0\\
 & & \cdots & & \\
0 & \cdots & 0 & 0 & \nu_{\ell_b - n + k}^{bjk}\\
0 & \cdots & 0 & 0 & 0 \end{array} \right) \end{equation}
and
\begin{equation*}
\beta_k^{bj,T} = \left( \begin{array}{cccccc}
\mu_1^{bjk} & 0 & 0 &  & \cdots & 0\\
0 & \mu_2^{bjk} & 0 & 0 & \cdots & 0\\
 & & \cdots & & &  \\
0 & \cdots & 0 & 0 & \mu_{\ell_b - n + k}^{bjk} & 0  \end{array} \right).
\end{equation*}
Similarly if \( \q \) is any quiver representing a point in \( Q_{[\sim,\mathcal{O}]}^{\circ,JCF} \)
whose Jordan blocks are of the form given by \( \alpha_k \) and \( \beta_k \) as
above, then the quiver \( \q^T \) obtained  from \( \q \) by replacing each such
Jordan block with the quiver given by \( \alpha_k^T \) and \( \beta_k^T \) 
satisfies the complex moment map equations for the action of \( H \), or
equivalently the complex moment map equations for the maximal
torus \( T_H \) of \( H \). 
\end{remark}

Recall from Definition~\ref{defhypertoric}
 the definition of \( M_T \) and \( \iota:M_T \hkq T_H \to Q \) inducing an
identification of the open subset \( Q_T^{\hks} = M_T^{\hks} \hkq T_H \)
of the hypertoric variety \( M_T \hkq T_H \) with its image in \( Q \). 
\begin{definition} \label{stratumhypertoric}
For any \( (\sim,\mathcal{O}) \) we can 
consider an open subset of a hyperk\"ahler modification
\begin{equation*}
((M_{[\sim,\mathcal{O}],T}\hkq T_H \times \mathbb{H} 
 ^\ell) \hkq (S^1)^\ell
\end{equation*}
(cf. Definition~\ref{hkmodification}) of the
hyperk\"ahler quotient \( M_{[\sim,\mathcal{O}],T}\hkq T_H 
 \) of 
the space \( M_{[\sim,\mathcal{O}],T} \)  of quivers which are direct sums of 
quivers of the form
\begin{equation} \label{diagquiv2}
 \alpha_k^T = \left( \begin{array}{ccccc}
\nu_1^k & 0 & 0 & \cdots & 0\\
0 & \nu_2^k & 0 & \cdots & 0\\
 & & \cdots & & \\
0 & \cdots & 0 & 0 & \nu_k^k\\
0 & \cdots & 0 & 0 & 0 \end{array} \right) \end{equation}
and
\begin{equation*}
\beta_k^T = \left( \begin{array}{cccccc}
\mu_1^k & 0 & 0 &  & \cdots & 0\\
0 & \mu_2^k & 0 & 0 & \cdots & 0\\
 & & \cdots & & &  \\
0 & \cdots & 0 & 0 & \mu_k^k & 0  \end{array} \right),
\end{equation*}
for \( \nu_i^k,\mu_i^k \in \C \),
with one such summand for every Jordan block of the canonical representative
\( \xi_0 \) of the nilpotent orbit \( \mathcal{O} \) in \( (\lie{k}_\sim)_\C \) as
in Remark \ref{remqt}. If
\( M_{[\sim,\mathcal{O}],T}^\circ  \) is the open subset of
\( M_{[\sim,\mathcal{O}],T} \) where all \( \nu_i^k \) and \( \mu_i^k \) are nonzero, then let 
\begin{equation*}
Q_{[\sim,\mathcal{O}],T}^\circ  = 
((M_{[\sim,\mathcal{O}],T}^\circ \hkq T_H \times \mathbb{H} \setminus \{ 0\})
 ^\ell) \hkq (S^1)^\ell.
\end{equation*}
Equivalently \( Q_{[\sim,\mathcal{O}],T}^\circ  \) can be identified with an open subset of the hyperk\"ahler quotient by  \( T_{{H}} \) of the space of quivers which, like
\( \q^T \) in Remark~\ref{remqt}, are direct sums of quivers of the
form above and summands
as in Remark \ref{remhypertoricstrat}.

The space \( M_{[\sim,\mathcal{O}],T} \)
 is a flat hypertoric variety
 with respect to the action of 
a quotient of \( T_{\tilde{H}} \). 
  Thus
\begin{equation*}
((M_{[\sim,\mathcal{O}],T}\hkq T_H) \times \mathbb{H} ^\ell) \hkq (S^1)^\ell
\end{equation*}
is also hypertoric, for the action of a quotient of \( T = T_{\tilde{H}}/T_H \), 
and it contains \( Q_{[\sim,\mathcal{O}],T}^\circ  \) as an open subset. 

We define \( Q_{[\sim,\mathcal{O}],T}  \) to be the open subset \( \SU(2) Q_{[\sim,\mathcal{O}],T}^\circ  \) of the hypertoric variety 
\(  
((M_{[\sim,\mathcal{O}],T}\hkq T_H \times \mathbb{H} 
 ^\ell) \hkq (S^1)^\ell.  \)
\end{definition}
\begin{definition}
Let \( \psi: Q_{[\sim,\mathcal{O}]}^{\circ,JCF} \to  Q_{[\sim,\mathcal{O}],T}^\circ   \)
be the map which associates to a quiver \( \q \in Q_{[\sim,\mathcal{O}]}^{\circ,JCF} \) the quiver \( \q^T \) described in Remark~\ref{remqt}.
Note that \( \psi \) is well defined, since any quiver which has the same form
as that for \( \q \) described in Remarks \ref{remsigma2} and \ref{remqt} and
which represents the same point in \( Q_{[\sim,\mathcal{O}]}^{\circ,JCF} \) (that is, lies in the same \( H \)-orbit) actually lies in the same \( T_H \)-orbit.
\end{definition}
\begin{lemma}
\( \psi: Q_{[\sim,\mathcal{O}]}^{\circ,JCF} \to  Q_{[\sim,\mathcal{O}],T}^\circ   \)
is a bijection.
\end{lemma}
\begin{proof}
Since \( \q \in Q_{[\sim,\mathcal{O}]}^{\circ,JCF}  \) satisfies the complex moment map equations for \( H \) and \( \alpha_{n-1} \beta_{n-1} \) is in Jordan canonical form, the entries \( \xi_i \in\C^* \) of the Jordan blocks of \( \q \) are uniquely
determined by the entries \( \nu_i, \mu_i \in\C^* \) of the corresponding blocks
of \( \q^T \).
\end{proof}

Recall from Remark~\ref{remchoice} that \( R_{[\sim,\mathcal{O}]} \) is the centraliser 
of the canonical representative \( \xi_0 \) of the nilpotent orbit \( \mathcal{O} \) in
 its Jacobson--Morozov parabolic \( P \) in \( (K_\sim)_\C \), while
\( [P,P] \cap R_{[\sim,\mathcal{O}]} \) stabilises each point of \(  Q_{[\sim,\mathcal{O}]}^{\circ,JCF} \). 
Thus \( P = Z_P [P,P] \) and  \( R_{[\sim,\mathcal{O}]} = (T_{[\sim,\mathcal{O}]})_\C ([P,P] \cap R_{[\sim,\mathcal{O}]}) \) where 
\( Z_P \) is the centre of the standard Levi subgroup of \( P \) and
\( (T_{[\sim,\mathcal{O}]})_\C = Z_P \cap  R_{[\sim,\mathcal{O}]}  \). 
Note also that \( R_{[\sim,\mathcal{O}]} \) contains the centre \( Z((K_\sim)_\C) \)
of \( (K_\sim)_\C \), which acts trivially on
\( Q_{[\sim,\mathcal{O}]}^\circ \) and on \( Q_{[\sim,\mathcal{O}],T}^\circ \),
 and that \( (K_\sim)_\C = Z((K_\sim)_\C)[(K_\sim)_\C,(K_\sim)_\C] \).
We obtain an immediate corollary as follows.

\begin{corollary} \label{corref}
\begin{equation*}
Q_{[\sim,\mathcal{O}]}^\circ \cong  K_\C \times_{R_{[\sim,\mathcal{O}]}  }
 Q_{[\sim,\mathcal{O}],T}^\circ  
=  (K_\C/[P,P] \cap  R_{[\sim,\mathcal{O}]}) \times_{ (T_{[\sim,\mathcal{O}]})_\C 
 }
 Q_{[\sim,\mathcal{O}],T}^\circ
\end{equation*}
\begin{equation*}
\cong K_\C \times_{(K_\sim)_\C} \left( ([(K_\sim)_\C,(K_\sim)_\C]/[P,P] \cap  R_{[\sim,\mathcal{O}]})  
\times_{(T^*_{[\sim,\mathcal{O}]})_\C} 
Q_{[\sim,\mathcal{O}],T}^\circ 
\right)
\end{equation*}
where  \(  Q_{[\sim,\mathcal{O}],T}^\circ  \) is an open subset of a hypertoric variety and \( (T^*_{[\sim,\mathcal{O}]})_\C
=(T_{[\sim,\mathcal{O}]})_\C \cap [(K_\sim)_\C,(K_\sim)_\C] \).
\end{corollary}
\begin{remark} \label{remref}
The quotient of \( K_\C/[P,P] \cap  R_{[\sim,\mathcal{O}]} \) by \( (T_{[\sim,\mathcal{O}]})_\C \) is of course the nilpotent orbit
\( K_\C/R_{[\sim,\mathcal{O}]} \) in \( \lie{k}_\C \) which contains
the nilpotent orbit
\begin{equation*}
\mathcal{O}= (K_\sim)_\C/R_{[\sim,\mathcal{O}]}
= [(K_\sim)_\C,(K_\sim)_\C]/R_{[\sim,\mathcal{O}]}\cap [(K_\sim)_\C,(K_\sim)_\C]
\end{equation*}
in \( (\lie{k}_\sim)_\C \)
or equivalently in \( [(\lie{k}_\sim)_\C ,(\lie{k}_\sim)_\C] \), which is
itself the quotient of 
\( [(K_\sim)_\C,(K_\sim)_\C]/[P,P] \cap R_{[\sim,\mathcal{O}]} \) by
\( (T^*_{[\sim,\mathcal{O}]})_\C \).
This nilpotent orbit for a product of special linear groups is 
an open subset of a hyperk\"ahler quotient of a flat hyperk\"ahler
space of quivers (cf. \cite{KS}), and \( K_\C/[P,P] \cap  R_{[\sim,\mathcal{O}]} \)
itself can likewise be described inductively in terms of the hyperk\"ahler
implosion of the corresponding product of special unitary groups.
\end{remark}

\begin{remark}  \label{remloctriv}
The hyperk\"ahler moment map \( \mathbb{H} \to \R^3 \) for the standard \( S^1 \)-action on \( \mathbb{H} \) restricts to a locally trivial fibration
\begin{equation*}
\mathbb{H} \setminus \{ 0 \} \to \R^3  \setminus \{ 0 \}
\end{equation*}
with fibre \( S^1 \). 
Recall the definition of the hypertoric variety \( M_T \) from \S3.
We can, as in \cite{BD},
 stratify \( M_T = \mathbb{H}^{n(n-1)/2} \) using the 
quaternionic coordinate hyperplanes, each stratum corresponding to fixing the
 subset \( E \) of \( \{1,\ldots,n(n-1)/2\} \) indexing the quaternionic hyperplanes
containing the points of the stratum.
 Then the hyperk\"ahler moment map \( M_T \to (\R^3)^{n(n-1)/2} \) restricted to
a stratum is a locally trivial fibration with fibre 
\begin{equation*}
T_{\tilde{H}}/(S^1)^{|E|}
\end{equation*}
where \( (S^1)^{|E|} \) is the subtorus of \( T_{\tilde{H}} = (S^1)^{n(n-1)/2} \)
whose Lie algebra is generated by the basis vectors indexed by the elements of \( E \).

Similarly
it follows from Definition~\ref{stratumhypertoric} and Remark~\ref{rem7.5a} that
that
\( \mu_{(S^1)^{n-1}}:  Q_{[\sim,\mathcal{O}],T} \to (\R^3)^{n-1}_\sim = 
(\lie{t}\otimes \R^3)_\sim \) is a locally trivial
fibration with fibre the quotient
\begin{equation*}
T/[P,P] \cap T_{[\sim,\mathcal{O}]}
\end{equation*}
of \( T = T_{\tilde{H}}/T_H \).
\end{remark}

\section{The refined strata}

Recall from (\ref{newstrat}) that the universal hyperk\"ahler implosion 
\( Q = M \hkq H \) for \( K=\SU(n) \) is a disjoint union
 \begin{equation*}
    Q = \coprod_{S,\delta, \sim} Q_{(S,\delta, \sim)} = \coprod_{\sim,\mathcal{O}} Q_{[\sim,\mathcal{O}]} 
  \end{equation*}
of subsets indexed by \( (S,\delta,\sim) \) or equivalently, as discussed
immediately before Definition~\ref{defnsimnil},
by pairs \( (\sim,\mathcal{O}) \) where
\( \sim \) is an equivalence relation on \( \{1,\ldots,n\} \) and \( \mathcal{O} \)
is a nilpotent adjoint orbit in 
 \( (\lie{k}_\sim)_\C \).
 Here 
\begin{equation*}
Q_{[\sim,\mathcal{O}]} = Q_{(S,\delta,\sim)} = Q_{(S,\delta)} \cap \mu_{(S^1)^{n-1}}^{-1}((\R^3)^{n-1}_{\sim})
\end{equation*}
as in Definition \ref{QSdeltasim}.
We have
\begin{equation*}
Q_{[\sim,\mathcal{O}]} = \SU(2) Q_{[\sim,\mathcal{O}]}^\circ
\end{equation*}
where \( Q_{[\sim,\mathcal{O}]}^\circ \) is the open subset of \( Q_{[\sim,\mathcal{O}]} \)
which is its intersection with
\begin{equation*}
\mu_{(S^1)^{n-1}}^{-1}( \{ (\lambda_1, \ldots, \lambda_{n-1}) \in (\R^3)^{n-1}_\sim: 
\sum_{k=i}^{j-1} \lambda_k = 0
  \ \text{in}\ \R^3
\end{equation*}
\begin{equation*}
\iff \sum_{k=i}^{j-1} \lambda_k^\C = 0
  \ \text{in}\ \C
\}
).
\end{equation*}
Recall also from Corollary~\ref{corref} and Remark~\ref{remref} that
\begin{equation*}
Q_{[\sim,\mathcal{O}]}^\circ =  K_\C \times_{R_{[\sim,\mathcal{O}]}  }
 Q_{[\sim,\mathcal{O}]}^{\circ,JCF}  
\cong  (K_\C/[P,P] \cap  R_{[\sim,\mathcal{O}]}) \times_{(T_{[\sim,\mathcal{O}]})_\C  }
 Q_{[\sim,\mathcal{O}],T}^\circ
\end{equation*}
where  \(  Q_{[\sim,\mathcal{O}],T}^\circ  \) is an open subset of a hypertoric variety, and \( R_{[\sim,\mathcal{O}]} \) is the centraliser in \( (K_\sim)_\C \) of
the standard representative \( \xi_0 \) in Jordan canonical form of the nilpotent
orbit \( \mathcal{O} \) in \( (\lie{k}_\sim)_\C \) (cf. Remark~\ref{remsigma2}), while
\( (T_{[\sim,\mathcal{O}]})_\C = T_\C \cap R_{[\sim,\mathcal{O}]}   \). Moreover
\( [P,P] \cap  R_{[\sim,\mathcal{O}]} \) acts trivially on 
\( Q_{[\sim,\mathcal{O}]}^{\circ,JCF}  \) and 
\( (T_{[\sim,\mathcal{O}]})_\C /[P,P] \cap (T_{[\sim,\mathcal{O}]})_\C  \)
acts freely on 
\( Q_{[\sim,\mathcal{O}]}^{\circ,JCF}  \) by Remark \ref{rem7.5a}.
In addition
 \( K_\C/[P,P] \cap  R_{[\sim,\mathcal{O}]} \)
can be described inductively in terms of the hyperk\"ahler
implosions of the special unitary groups whose product is
\(  [(K_\sim)_\C,(K_\sim)_\C] \).
Recall finally from  Lemma~\ref{leminj} and Remark~\ref{remsigma2} that 
the non-empty fibres of the restriction
\begin{equation*}
Q_{[\sim,\mathcal{O}]}^{\circ,JCF} \to \lie{k}_\C \oplus \lie{t}_\C
\end{equation*}
to \( Q_{[\sim,\mathcal{O}]}^{\circ,JCF}  \) of the complex moment map
for the action of \( K \times T \) on \( Q \) are single \( (T_{[\sim,\mathcal{O}]})_\C \times T_\C \)-orbits,  and its image is 
\begin{equation*}
\Delta (\lie{t}_\C)_\sim \oplus \xi_0 = \{ (\zeta + \xi_0, \zeta) \in \lie{k}_\C \oplus \lie{t}_\C : \zeta \in (\lie{t}_\C)_\sim \}.
\end{equation*}
Putting this all together we obtain the following theorem. 

\begin{theorem} \label{thm6.8}
For each 
equivalence relation \( \sim \) on \( \{1,\ldots,n\} \) and nilpotent adjoint orbit \( \mathcal{O} \)
for \( (K_\sim)_\C \), 
 the stratum \(  Q_{[\sim,\mathcal{O}]} \) 
is the union over \( s \in \SU(2) \) of its open subsets 
 \( sQ_{[\sim,\mathcal{O}]}^\circ \), and 
\begin{equation*}
Q_{[\sim,\mathcal{O}]}^\circ =  K_\C \times_{R_{[\sim,\mathcal{O}]}  }
 Q_{[\sim,\mathcal{O}]}^{\circ,JCF}
\end{equation*}
where \( R_{[\sim,\mathcal{O}]} \) is the centraliser in \( (K_\sim)_\C \) of
the standard representative \( \xi_0 \) in Jordan canonical form of the nilpotent
orbit \( \mathcal{O} \) in \( (\lie{k}_\sim)_\C \), and  \(  Q_{[\sim,\mathcal{O}]}^{\circ,JCF}  \) can be identified with an open subset of a hypertoric variety. The image of  
 the restriction
\begin{equation*}
Q_{[\sim,\mathcal{O}]}^{\circ} \to \lie{k}_\C
\end{equation*}
 of the complex moment map for the action of \( K \) on \( Q \) is 
\( K_\C( (\lie{t}_\C)_\sim \oplus \mathcal{O}) \cong K_\C \times_{(K_\sim)_\C}
( (\lie{t}_\C)_\sim \oplus \mathcal{O})
 \) 
and its fibres are single \( (T_{[\sim,\mathcal{O}]})_\C \times T_\C \)-orbits, where
\( (T_{[\sim,\mathcal{O}]})_\C = T_\C \cap R_{[\sim,\mathcal{O}]}   \) and
\( (T_{[\sim,\mathcal{O}]})_\C /[P,P] \cap (T_{[\sim,\mathcal{O}]})_\C  \)
acts freely on 
\( Q_{[\sim,\mathcal{O}]}^{\circ,JCF}  \).
\end{theorem}
\begin{remark}
Recall  that the symplectic implosion \( X_{{\rm impl}} \) of  a symplectic manifold \( X \) with a Hamiltonian action of a compact group \( K \) with moment
map \( \mu_X \) is the 
disjoint union over the faces \( \sigma \) of a positive Weyl chamber \( \lie{t}_+ \)
of the strata
\begin{equation*}
\mu_X^{-1}(\sigma)/[K_\sigma,K_\sigma]
\end{equation*}
where \( K_\sigma \) is the stabiliser in \( K \) of any element of the face \( \sigma \)
of \( \lie{t}_+ \) and \( [K_\sigma,K_\sigma] \) is its commutator subgroup.
Similarly if \( X \) is a hyperk\"ahler manifold with a hyperk\"ahler action
of \( K=\SU(n) \) and hyperk\"ahler moment map \( \mu_X:X \to \lie{k} \otimes \R^3 \)
and complex moment map its projection \( \mu_{X,\C} \) to \( \lie{k}_\C \), then the hyperk\"ahler implosion \( X_{{\rm hkimpl}} \) of \( X \) is defined to be the hyperk\"ahler
quotient of \( X \times Q \) by the diagonal action of \( K \). In the light of 
Theorem \ref{thm6.8} and Remark~\ref{rem7.5a} we expect to have a description of
\( X_{{\rm hkimpl}} \) as follows, at least when \( X \) is an affine variety with
respect to all its complex structures so that symplectic quotients can be identified with GIT quotients.
 \( X_{{\rm hkimpl}} \) should be the disjoint union 
over all equivalence relations \( \sim \) on \( \{1,\ldots,n\} \) and nilpotent adjoint
orbits in \( (\lie{k}_\sim)_\C \)
of its
subsets
\( X_{{\rm hkimpl} [\sim,\mathcal{O}] } \), 
which are themselves the unions of open subsets
\( sX_{{\rm hkimpl} [\sim,\mathcal{O}] }^\circ  \) for \( s \in \SU(2) \), 
 such that
\begin{equation*}
X_{{\rm hkimpl} [\sim,\mathcal{O}] }^\circ  = 
\mu_{X,\C}^{-1}((\lie{t}_\C)_\sim \oplus \xi_0)/
[P,P] \cap  R_{[\sim,\mathcal{O}]}
\end{equation*}
where \(  \xi_0 \in \mathcal{O} \) is the canonical representative of the nilpotent orbit
\( \mathcal{O} \subseteq (\lie{k}_\sim)_\C \) in Jordan canonical form as in 
Remark~\ref{rem7.5a}, and 
\( (\lie{t}_\C)_\sim \) is the set of elements of \( \lie{t}_\C \)
with centraliser \( K_\sim \) in \( K \),
while \( P \) is the Jacobson-Morozov parabolic of
\( \xi_0 \) in \( (K_\sim)_\C) \) and \( R_{[\sim,\mathcal{O}]} \) is the centraliser  of
\( \xi_0 \) in \( (K_\sim)_\C  \). 

\end{remark}

\section{An approach via Nahm's equations}

\bigskip 
The results so far have been proved for the case \( K=SU(n) \), where 
from \cite{DKS} we
have a finite-dimensional description of the universal hyperk\"ahler implosion
in terms of quivers. However in the current paper we have tried
to formulate many of our results in a way that could potentially
admit generalisation to other compact groups.
In this final section we sketch a gauge-theoretic approach,
involving Nahm's equations, which could provide another means of
attacking this problem.
We recall that the Nahm equations 
are the system
\begin{equation} \label{Nahm}
  \frac{dT_i}{dt} + [T_0 , T_i] = [T_j, T_k], \qquad \text{\( (ijk) \)
  cyclic permutation of \( (123) \),}
\end{equation}
where \( T_i \) take
values in \( \kf \) and are smooth on some specified interval \( \mathcal{I} \). Moduli
spaces of solutions to the Nahm equations are
 obtained by quotienting by the gauge action
\begin{equation} \label{gauge}
  T_0 \mapsto g T_0 g^{-1} - \dot{g} g^{-1},\qquad 
  T_i \mapsto g T_i g^{-1} \ (i=1,2,3),
\end{equation}
where \( g \colon \mathcal{I} \mapsto K \), subject to appropriate
constraints on \( g \). The Nahm equations
may be interpreted as the vanishing condition for a
hyperk\"ahler moment map for the action (\ref{gauge}) of the group of gauge
transformations on an infinite-dimensional flat quaternionic space of
\( \kf \)-valued functions on the interval \( \mathcal I \). In this way 
Nahm moduli spaces can acquire a hyperk\"ahler structure. 
In particular Kronheimer \cite{Kronheimer:semi-simple}, Biquard \cite{Biquard}
and Kovalev \cite{Kovalev} have shown that coadjoint orbits
of \( K_{\C} \) may be given hyperk\"ahler structures as moduli spaces of Nahm data on the half-line \( [0,\infty) \), while Kronheimer \cite{Kronheimer:cotangent} has shown that 
the cotangent bundle 
of \( K_{\C} \) may be given a hyperk\"ahler structure as a moduli space of Nahm data on 
the interval \( [0,1] \).
\medskip
Let us fix a Cartan algebra \( \tf \) of the Lie algebra \( \kf \) of the compact group
\( K \).  We can consider quadruples \( (T_0,T_1, T_2,T_3) \), where each \( T_i \) takes
values in \( \kf \), satisfying the Nahm equations and defined on the half line
\( [0, \infty) \). We recall from \cite{Biquard} that such a solution has asymptotics
\begin{equation*}
  T_i = \tau_i + \frac{\sigma_i}{t} + \dotsb \qquad (i=1,2,3),
\end{equation*}
where \( \tau=(\tau_1, \tau_2, \tau_3) \) is a commuting triple, which we shall take
to lie in the fixed Cartan algebra \( \tf \). Also \( \sigma_i = \rho(e_i) \) where
\( e_1,e_2,e_3 \) is a standard basis for \( \su(2) \) and \( \rho \colon \su(2)
\rightarrow \kf \) is a Lie algebra homomorphism, so \( [\sigma_1, \sigma_2] =
\sigma_3 \) etc.  Moreover we must have \( [\tau_i, \sigma_j]=0 \) for
\( i,j=1,2,3 \); equivalently, \( \rho \) takes values in the Lie algebra
\( \cf \) of the common centraliser \( C(\tau_1,\tau_2,\tau_3) \) of the triple \( (\tau_1,\tau_2,\tau_3) \).

We factor out by gauge transformations equal to the identity at \( 0,\infty \). In addition,
we have an action of \( T \) by gauge transformations that are the identity at
\( t=0 \) and take values in \( T \) at \( t=\infty \).
If the common centraliser \( C(\tau_1,\tau_2, \tau_3) \) of the triple
\( (\tau_1, \tau_2, \tau_3) \) is just the maximal torus \( T \), then in
fact \( \rho \) and hence the \( \sigma_i \) are zero. The resulting Nahm
moduli space, where we quotient by \( T \), is exactly Kronheimer's
description of a semisimple orbit as a moduli space
\cite{Kronheimer:semi-simple}.

If the centraliser of the 
triple is larger, we must consider Nahm data asymptotic to the
triple, but with various possible \( \frac{\sigma_i}{t} \) terms.
 The various coadjoint orbits are obtained by choosing
\( \sigma_i \) and factoring out by gauge transformations that are \( I \) at \( t=0 \)
and lie in \( C(\tau_1, \tau_2, \tau_3) \cap
C(\sigma_1, \sigma_2, \sigma_3) \) at infinity. The coadjoint
orbits will fit
together to form the Kostant variety corresponding to our choice of \( \tau \).
(We recall that the Kostant varieties are the varieties obtained by
fixing the values of a generating set of invariant polynomials
\cite{Kostant:polynomial}).
The semisimple stratum will correspond to \( \sigma_i=0 \).

\medskip
The universal hyperk\"ahler implosion for an arbitrary compact group \( K \) 
is expected to be a space, which, as in the finite-dimensional
\( SU(n) \) quiver picture of the preceding section, admits a
hyperk\"ahler torus action, with hyperk\"ahler reductions giving the
Kostant varieties.

In terms of Nahm data, this should mean that to obtain
the  universal hyperk\"ahler implosion we do not fix triples \( \tau \)
but allow them to vary in a fixed Cartan algebra, and that we should factor
out only gauge transformations asymptotic to the identity as \( t \) tends to \( \infty \), so the above \( T \) action remains.
The moment map for this should formally be evaluation of \( (T_1, T_2, T_3) \) at
\( \infty \); that is, it should give the triple \( (\tau_1, \tau_2, \tau_3) \).

Then the hyperk\"ahler quotient by \( T \) would be the space of Nahm data on
\( [0, \infty) \) asymptotic to a \emph{fixed} triple \( (\tau_1, \tau_2,
\tau_3) \), modulo gauge transformations equal to the identity at \( t=0 \) and
\( T \)-valued at infinity. As mentioned above, if the triple  has common
centraliser \( T \), 
this exactly gives the corresponding Kostant variety, which in this case is just the
regular semisimple orbit.

For a general commuting triple \( \tau = (\tau_1, \tau_2, \tau_3) \), the Kostant variety is stratified by different
orbits, which as above are obtained by fixing \( \tau \)
and \( \rho \) (or equivalently \( \sigma \)) and factoring by gauge transformations that take values
in the common centraliser \( C(\tau, \sigma) \) of \( \tau \) and \( \sigma \) at infinity.
To obtain this Kostant variety via hyperk\"ahler reduction by \( T \), we need
therefore to perform further collapsings on the moduli space.  More
precisely, we should collapse by factoring out gauge transformations that are the
identity at \( t=0 \) but take values at \( t= \infty \) in the commutator
\( [C,C] \), where \( C=C(\tau_1, \tau_2,
\tau_3) \).  So on the open
dense set of the moduli space where the triple \( (\tau_1, \tau_2,
\tau_3) \) has centraliser \( T \), no collapsing occurs. In general, reducing
by \( T \) will lead to quotienting by \( C = C(\tau_1, \tau_2,
\tau_3) \).
The action of \( C \) can be used to bring the \( \sigma_i \) into one of
a finite list of standard forms (one for each stratum of the Kostant variety
corresponding to the choice of \( \tau \)),
and then the remaining freedom lies in \( C(\tau, \sigma) \), which is what we
need to factor out by to get the coadjoint orbit.

We denote the space obtained via this collapsing by \( {\mathcal Q} \),
and this is a candidate for the universal hyperk\"ahler implosion
for the general compact group \( K \). We can
stratify \( \mathcal Q \) by the centraliser \( C(\tau_1, \tau_2, \tau_3) \) of the triple \( (\tau_1, \tau_2,
\tau_3) \). That is, for each compact subgroup \( C \) of \( K \), we consider
\( {\mathcal Q}^{C} \), the space of Nahm triples with \( C(\tau_1,
\tau_2, \tau_3) = C \), modulo gauge transformations which are the identity
at \( t=0 \) and take values in \( [C,C] \) at \( t= \infty \).  The top stratum is
then the open dense set where the centraliser of the 
triple is the maximal torus.
This stratification agrees in the case \( K=\SU(n) \) with the stratification by 
\( K_{\sim} \) in the quiver picture \( Q=M \hkq H \), as in Remark \ref{remksim}; if we stratify 
further using \( \sigma \) we can obtain strata corresponding to the subsets \( Q_{[\sim,\mathcal{O}]} \) of \( Q \).
Notice that the regularity condition (\ref{reg})
(which may always be achieved
by a generic \( SU(2) \) rotation) is the condition of \emph{Biquard regularity}
\cite{Biquard}
for the triple \( (\tau_1, \tau_2, \tau_3) \), that is \( C(\tau_1, \tau_2, \tau_3)
= C(\tau_2, \tau_3) \).

\begin{remark}
 This is analogous to the symplectic case, where we obtain the implosion
  by taking \( K \times \tf_{+}^* \) and collapsing by commutators of points in
  the Weyl chamber \( \tf_{+}^* \), so that in the interior of the chamber no
  collapsing takes place.
\end{remark}

In the symplectic case each stratum can be viewed as a symplectic quotient
of a suitable space by the commutator, and hence is itself symplectic.
There is an analogous statement in the hyperk\"ahler setting.
For we have a decomposition
\begin{equation*}
  \cf^* = \z(\cf)^* \oplus [\cf^*, \cf^*]
\end{equation*}
where \( \cf \) is the Lie algebra of the common centraliser \( C \) of the triple \( (\tau_1,\tau_2,\tau_3) \).
Now consider the space of Nahm solutions with \( T_i(\infty) \in \cf^*
(\mbox{for }i=0,1,2,3) \), modulo gauge transformations equal to the identity at \( 0,
\infty \).
There is a \( C \) action on this space by gauge transformations equal to the
identity at \( t=0 \) and lying in \( C \) at \( t=\infty \).  The formal 
moment map for this
action is evaluation at \( \infty \). So the moment map for the \( [C,C] \) action
is evaluation at \( \infty \) followed by projection onto \( [\cf^*, \cf^*] \).
So the formal hyperk\"ahler quotient by \( [C,C] \) at level zero is
the set of Nahm matrices with \( T_{i}(\infty) \in \z(\cf^*) : i=1,2,3 \),
modulo the action of \( [C,C] \). The stratum \( {\mathcal Q}^{C} \),
that is, the quotient by \( [C,C] \) of
the set of Nahm matrices with common centraliser of \( T_1(\infty),
T_2(\infty), T_3(\infty) \) equal to \( C \), is then an open
dense subset of this quotient.

\medskip It should also be possible to construct a hypertoric
variety
by considering the sweep under the \( T \) action of the set of constant solutions
\( (0, c_1, c_2, c_3) \; : \; c_i \in \lie{t} \) to the Nahm equations. On the
open subset where no
collapsing takes place the \( T \) action is free.

\bigskip
The above statements are formal --- in order
to carry out the programme mentioned at the beginning
 of this section
we need to provide an analytical
framework. In particular we need to work out a suitable
  stratified hyperk\"ahler metric. We conclude by discussing some of the
issues in defining such a metric.

\medskip
As above,
we fix a Cartan algebra \( \tf \) of the Lie algebra \( \kf \) of the compact group
\( K \).  We consider quadruples \( (T_0,T_1, T_2,T_3) \), where each \( T_i \) takes
values in \( \kf \), satisfying the Nahm equations and defined on the half line
\( [0, \infty) \). 
We form a moduli space \( \tilde{\mathcal M} \) from the above data by
quotienting out by gauge transformations that are the identity at \( t=0,
\infty \).

The standard \( L^2 \) metric on Nahm moduli spaces over the interval \( [0,
\infty) \) is given by
\begin{equation*}
  \parallel (X_0,X_1, X_2, X_3) \parallel^2 = \int_{0}^{\infty} 
  \sum_{i=0}^{3} \langle X_i , X_i \rangle \, dt
\end{equation*}
for tangent vectors \( X=(X_0,\dots, X_3) \), where \( \langle, \rangle \) denotes
the Killing form on \( \kf \).

Write our tangent vector as
\begin{equation*}
  X_i = \delta_i + \frac{\epsilon_i}{t} +  \ldots \qquad (i=1,2,3),
\end{equation*}
so
\begin{equation*}
  \langle X_i, X_i \rangle = \langle \delta_i, \delta_i \rangle
  + 2 \frac{\langle \delta_i , \epsilon_i \rangle}{t} + O(\frac{1}{t^r}) 
\qquad(r>1)
\end{equation*}
The \( L^2 \) metric will not be finite except in tangent directions where
\( \delta_i=0 \).  Such directions correspond to those tangent to the Nahm
matrices with a fixed commuting triple \( (\tau_1, \tau_2, \tau_3) \).

We may, however, modify our metric following Bielawski \cite{Bielawski:hyper-kaehler}
thus:
\begin{equation*}
  \begin{split}
   \lvert (X_0,X_1, X_2, X_3) \rvert^2 &= \int_{0}^{\infty} \sum_{i=0}^{3}
   (\langle X_i , X_i \rangle -\langle X_i(\infty), X_i(\infty) \rangle \,
   ) dt\\
    &\quad+ c \langle X_i(\infty), X_i(\infty) \rangle,
  \end{split}
\end{equation*}
where \( c \) is a constant.
This defines a symmetric bilinear form, though definiteness and nondegeneracy
properties remain unclear.  
Now \\ \( X_i(\infty) = \delta_i \), so we see that the
Bielawski pseudometric 
is finite in directions such that \( \langle \delta_i,
\epsilon_i \rangle =0 \) for all \( i \).  So it is finite even in certain
directions that correspond to infinitesimally changing the \( \tau_i \).

In particular, it is finite on the open dense set of \( \tilde{M} \) consisting
of all Nahm matrices where the triple \( (\tau_1, \tau_2, \tau_3) \) is
regular, since, as remarked above, the \( \epsilon_i \) terms will be zero for
directions tangent to this region.

To analyse which directions are finite in general we need to use the
relation \( [ \tau_i, \sigma_j]=0 \), \( (i,j,=1,2,3) \). Differentiated, this
gives the relation
\begin{equation}
  \label{eq:commrelation}
  [\delta_i, \sigma_j] + [\tau_i, \epsilon_j]=0 \qquad (i,j=1,2,3)
\end{equation}
on tangent vectors.

Suppose we stratify Nahm data by the centraliser \( C \) of
\( (\tau_1,\tau_2,\tau_3) \). 
Let us consider a tangent vector to a
stratum.

For all elements in the stratum, and for all
\( h \in \cf = {\rm Lie} \;(C) \), we have \( [\tau_i, h]=0 \).
 It follows that for our tangent vector,
\begin{equation}
  \label{eq:deltacomm}
  [\delta_i, h]=0 \qquad\forall h \in \cf 
\end{equation}
In particular,
\begin{equation*} [\delta_i, \sigma_j]=0 \qquad (i,j=1,2,3).
\end{equation*}
Now our above relation \eqref{eq:commrelation} shows
\begin{equation*} [\tau_i, \epsilon_j]=0 \qquad (i,j=1,2,3)
\end{equation*}
so \( \epsilon_j \in \cf \).
If \( \epsilon_i \) is a commutator \( [\xi, \eta] \), where \( \xi, \eta \in \cf \),
we have:
\begin{align*}
  \tr \delta_i \epsilon_i
  &= \tr \delta_i \xi \eta - \tr \delta_i \eta \xi \\
  &= \tr \delta_i \xi \eta - \tr \xi \delta_i \eta \\
  &= 0
\end{align*}
where in the last step we have used the above observation that \( [\xi,
\delta_i]=0 \) since \( \xi \in \cf \).
By linearity, this also holds if \( \epsilon_i \) is a sum of commutators in
\( \cf \).

But recall that \( [\sigma_1, \sigma_2]=\sigma_3 \) etc. So, passing to tangent
vectors,
\begin{equation*} [\sigma_1, \epsilon_2] + [\epsilon_1, \sigma_2] =
 \epsilon_3
\end{equation*}
and cyclically. Each \( \epsilon_i \) is indeed therefore a sum of commutators
in \( \cf \), and hence we see \( \langle \delta_i,\epsilon_i \rangle =0 \).  We
deduce

\begin{theorem}
  The Bielawski pseudometric is finite in directions tangent to the set of Nahm
  matrices with fixed centraliser of \( (\tau_1, \tau_2, \tau_3) \).
\end{theorem}

\begin{remark}
 If we stratify Nahm matrices by centraliser of the triple, we therefore
 obtain a metric in the stratified sense. Note that the top stratum for
this stratification is just the set of Nahm matrices where the 
centraliser of the triple is the maximal torus.
 
\end{remark}

\begin{remark}
  In fact, the above calculation works provided \( [\delta_i, \sigma_j]=0 \)
  for all \( i,j \), as we know from above that \( \epsilon_i \) is a sum of
 commutators \( [\xi_k, \eta_k] \) with \( \xi_k = \sigma_k \). In particular if
  \( \tau_i=0 \) for all \( i \) then by \eqref{eq:commrelation} this condition
  holds, so the metric is finite on all tangent vectors, not just those
  tangent to the stratum.

  If \( K=\SU(2) \) we only have two strata, the regular one and the stratum
  where all \( \tau_i \) are zero. It follows that the metric is everywhere
  finite in this case, which checks as  the
  implosion is now flat \( \HH^2 \).
\end{remark}

\begin{remark}
  Here we have been considering the Nahm equations on the half-line \( 
  [0,\infty) \), motivated by the constructions of coadjoint orbits of
  \( K_\C \) in Kronheimer \cite{Kronheimer:semi-simple}, Biquard \cite{Biquard} and
  Kovalev \cite{Kovalev}. However we expect the universal
  hyperk\"ahler implosion for \( K \) to be the 
 complex-symplectic GIT reduction by the maximal
unipotent group \( N \) of the cotangent bundle \( T^*
  K_{\C} \) of \( K_{\C} \), and in \cite{Kronheimer:cotangent} \( T^* K_{\C} \)
  is given a hyperk\"ahler structure as a moduli space of Nahm data on
  the interval \( [0,1] \). Comparison of this construction with the
  description in the last section of a hyperk\"ahler implosion
  \( X_{{\rm hkimpl}} \) when \( K=\SU(n) \) suggests a formal picture of the
  universal hyperk\"ahler implosion for \( K \) which is similar to that
  above but uses Nahm data on the interval \( [0,1] \) instead of
  \( [0,\infty) \).
\end{remark}

\end{document}